\documentclass[preprint,11pt]{article}
\usepackage{amsmath,amssymb}
\usepackage{mathrsfs}
\usepackage{subfig}
\usepackage{graphicx}
\usepackage{algorithm}
\usepackage{algpseudocode}
\usepackage{color}
\usepackage{textcomp}
\usepackage[table]{xcolor} 
\usepackage{tabularx}
\usepackage{upgreek}
\usepackage{fullpage}
\usepackage{url}

\begin{document}

\newtheorem{theorem}{Theorem}[section]
\newtheorem{lemma}[theorem]{Lemma}
\newtheorem{proposition}[theorem]{Proposition}
\newtheorem{corollary}[theorem]{Corollary}

\newenvironment{proof}[1][Proof]{\begin{trivlist}
\item[\hskip \labelsep {\bfseries #1}]}{\end{trivlist}}
\newenvironment{definition}[1][Definition]{\begin{trivlist}
\item[\hskip \labelsep {\bfseries #1}]}{\end{trivlist}}
\newenvironment{example}[1][Example]{\begin{trivlist}
\item[\hskip \labelsep {\bfseries #1}]}{\end{trivlist}}
\newenvironment{remark}[1][Remark]{\begin{trivlist}
\item[\hskip \labelsep {\bfseries #1}]}{\end{trivlist}}

\newcommand{\qed}{$\square$}
\newcommand{\comment}[1]{\textcolor{red}{#1}}
\newcommand{\wt}[1]{\widetilde{#1}}
\newcommand{\jump}[1]{\text{\textlbrackdbl}#1\text{\textrbrackdbl}}
\makeatletter
\newcommand*{\rom}[1]{\expandafter\@slowromancap\romannumeral #1@}
\newcommand{\perf}{\widetilde{\Omega}}
\newcommand{\solid}{{\cal S}}
\newcommand{\restrict}{{\cal R}}
\newcommand{\extend}{{\cal E}}
\newcommand{\Hspace}{H^1_{0}(\Omega)}
\newcommand{\Hperf}{H^1_{D}(\perf)}
\newcommand{\Qint}{{\cal I}_{H}}
\newcommand{\Qintp}{\widetilde{{\cal I}}_{H}}
\newcommand{\norm}[1]{\left\lVert#1\right\rVert}
\newcommand\abs[1]{\left|#1\right|}
\newcommand{\eps}{\varepsilon}
\newcommand{\de}{\delta}
\newcommand{\local}{C_{{\cal I}_{H}}}
\newcommand{\localprime}{{C'}_{{\cal I}_{H}}}

\newcommand{\avrg}[2]{{\langle #1 \rangle}_{#2}}
\newcommand{\TwoNorm}[2]{{\left\lVert#1\right\rVert}_{L^2(#2)}}

\makeatother
%
\title{A Multiscale Method for   Porous Microstructures}

\author{Donald L. Brown\thanks{Institute for Numerical Simulation, University of Bonn,  brown@ins.uni-bonn.de}\and Daniel Peterseim\thanks{Institute for Numerical Simulation, University of Bonn}}




\maketitle

\begin{abstract}
In this paper we develop a multiscale method to solve  problems in complicated porous microstructures with Neumann boundary conditions. By using a  coarse-grid quasi-interpolation operator to define a fine detail space and local orthogonal decomposition, we   construct multiscale corrections to coarse-grid basis functions with microstructure. 
By truncating the corrector functions we are able to make a computationally efficient scheme. Error results and analysis are presented.
A key component of this analysis is the investigation of the Poincar\'{e} constants in perforated domains as they may contain micro-structural information. Using a constructive method originally developed for weighted Poincar\'{e} inequalities, we are able to obtain estimates on Poincar\'{e} constants with respect to scale and separation length of the pores. Finally, two numerical examples are presented to verify our estimates. 
\end{abstract}



\section{Introduction}

 Modeling and simulation of porous media has many wide ranging applications in engineering. For example, to simulate heat  or electric conductivity in complicated materials or composites a partial differential equation (PDE) in complicated  microstructures must be solved.  Direct numerical  simulation of such problems   is difficult, and, in some scenarios is intractable. The main challenge being the many scale nature of the problem and complex geometries  involved.  In these applications, where there are many scales and complex heterogeneities, numerical homogenization procedures are employed to reduce complexity yet remain accurate. In this work, we develop a multiscale method to simulate Neumann problems in domains with porous microstructures. 


The study of multiscale problems in porous or  perforated domains has a long history.  In the area of homogenization of partial differential equations,  there is a vast literature on the subject \cite{ref:Chechkin.2007,ref:Paloka&Mikelic.1996,ref:Sanchez-Palencia.1980} and references therein, to name just a few. In these problems, the fine-scale equations have microstructure, then through an averaging process of homogenization an effective PDE is derived. In these methods, the strong assumption of periodicity is usually made, and thus, only one microstructure dependent local problem is solved to compute effective properties. The coarse-grid, or homogenized problem, does not have explicit microstructure. More computationally based procedures have also been investigated. Using an approach based on the Heterogeneous Multiscale Method  \cite{Abdulle:E:Engquist:Vanden-Eijnden:2012}, an algorithm was developed in \cite{Henning_theheterogeneous} by solving for an unknown diffusion coefficient on the coarse-grid by resolving a local perforated domain problem. Then, computation on the coarse-grid equation is based in an effective non-porous domain.  Further  work, \cite{bris:hal-00841308}, developed a perforated multiscale finite element method for Dirichlet problems utilizing  Crouzeix-Raviart non-conforming finite elements. Using the MsFEM framework \cite{ref:Efendiev&Hou.2009},  a weak Crouzeix-Raviart boundary condition is used to construct the multiscale finite element basis that include the vanishing Dirchlet condition into the basis functions. 
There are also mesoscopic schemes that relax the resolution condition of standard
finite elements insofar as they allow that mesh cells are cut by the domain
boundary;
see e.g. \cite{MR752608,MR2567257,MR3200279,MR1431211,MR2496016} among many others. However,
for those schemes there are typically strong restrictions on the topology of the
intersection that rule out the case of perforation on the element level.

We will work in the multiscale framework using a local orthogonal decomposition (LOD) \cite{MP11}, which is
inspired by the variational multiscale method \cite{MR1660141,MR2300286,Malqvist:2011}.
The LOD method uses a coarse-grid quasi-interpolation operator  to decompose the space into fine-scale components to build the fine detail space. From the fine detail space we are able to build multiscale corrections to the coarse-grid functions and construct a multiscale space. These corrections  have global support, thus limiting their practical usage. However, these corrections have fast decay and can therefore be localized. This procedure has been used effectively for elliptic problems with $L^{\infty}$ coefficients \cite{Henning.Morgenstern.Peterseim:2014,HP13,MP11}, been extended to semi-linear elliptic equations \cite{HMP12}, linear and nonlinear eigenvalue problems \cite{HMP13,MP12}, and to the wave equation \cite{Henningwave}.


In this work, we extend this framework  to the case when we have microstructures that generate the multiscale features as opposed to oscillatory and highly varying coefficients. 
We first build a coarse-perforated grid, then by using a quasi-interpolation operator based on local $L^2$ projection build a fine-scale space. We again follow the process in \cite{Henning.Morgenstern.Peterseim:2014,HP13,MP11} of multiscale space construction, localization, and subsequent error estimates. We show that we can obtain the same error estimates with respect to coarse-grid size and truncation of local problems as in these works. However, in this setting we are particularly concerned with the tracking of Poincar\'{e} constants in perforated domains as these may depend on the micro-structural features, namely the size of particles and separation length. Using the methods developed in \cite{Pechstein01042013}, originally for the setting of high-contrast coefficients and weighted Poincar\'{e} inequalities, we are able to create a constructive procedure to estimate these constants in domains with microstructure. This is carried out for a few interesting examples. We show that in the case of a reticulated filamented structure it is possible that the microstructural features can negatively impact this Poincar\'{e} constant in the case of very thin structures. In addition, we show that in the case of isolated particles we obtain uniform (microstructure independent) Poincar\'{e} constants.

The paper is organized as follows. We begin by the problem setting and the description of quasi-interpolations in perforated domains. This quasi-interpolation will allow us to construct our multiscale orthogonal splitting and subsequent computational localization algorithm. Then, we will derive error estimates on both global supported and localized basis functions. This is done with the help of technical lemmas in the Appendix and careful tracking of relevant constants. We then develop a constructive procedure to estimate Poincar\'{e} constants in porous domains. Finally, we give two numerical examples to demonstrate the rates of convergence with respect to mesh parameters, localization truncation, and microstructure lengths. In addition, we discuss overall effectiveness of the algorithm and the choices of possible quasi-interpolation operators.

\section{Problem Set Up}
 We now begin with some notation and problem setting. Let $\Omega\subset \mathbb{R}^d$ be a bounded Lipschitz domain with polyhedral boundary for $d\geq 2$. We denote the solid microstructure to be  $\{S_{i}\}_{i=1}^N$, a set of Lipschitz nonintersecting closed subsets of $\Omega$.  We denote the perforated domain, often called fluid or porous domain, 
 $\perf=\Omega \backslash \solid $, where $\solid=\cup_{i=1}^N S_{i}$. We supposed that the solid microstructure or inclusions are so that $\perf$ remains connected and Lipschitz.
 We let $\eta$ be the characteristic size of the microstructure.
  \begin{figure}[!ht]
  \centering
      \includegraphics[width=0.5\textwidth]{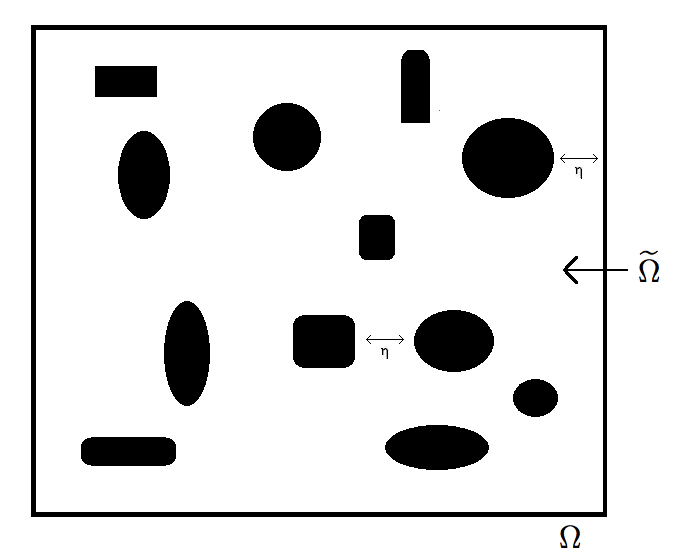}
       \caption{Domain $\Omega$ with microstructure. The porous part of the domain is denoted $\wt{\Omega}$.}
    \label{domain}
  \end{figure}
Moreover, we let $\eta$ also be the minimal separation length. These two parameters may be considered separately, but for clarity we choose them to be on the same order of magnitude. We suppose for simplicity that the perforations do not intersect the global boundary, but may be $\eta$ close to it.  An example geometry can be seen in Figure \ref{domain}.

We wish to find a solution $u$ that satisfies 
\begin{subequations}\label{perfLaplace}
\begin{align}
 -\Delta u&=g \text{ in } \perf,\\
 \label{neumann}
 \frac{\partial u}{\partial n}&=0 \text{ on } \partial \solid ,\\
 u&=0 \text{ on } \partial \Omega.
\end{align}
\end{subequations}
Where $g\in L^2(\wt{\Omega})$, and $n$ denotes the outer normal on $\partial \solid.$

We denote the space $H^1_{D}(\perf):=\{v\in H^1(\perf)\;|\;v=0 \text{ on } \partial \perf\}$. Multiplying by $v\in \Hperf$ and  integrating \eqref{perfLaplace}, we wish to solve for $u\in \Hperf$ such that  
\begin{align}\label{varform}
 \int_{\perf}\nabla {u} \nabla v dz=\int_{\perf} g v dz,
\end{align}
here $dz$ is the standard real Lebesgue measure in $\mathbb{R}^d$. The main difficulty in solving the above problem is the mutliscale nature introduced from the microstructure. We may also add in an oscillatory coefficient inside the perforated domain $\perf$, however, this case is well studied in \cite{MP11} and we focus on the issues involved with the multiscale geometries.

\section{Quasi-Interpolation in Perforated Domains}
In this section we develop the framework to work on perforated domains. We first define  the classical nodal basis  restricted  to $\perf$ the perforated domain. Then, we describe how to construct a quasi-interpolation that is also projective, in contrast to the quasi-interpolation operator used in  \cite{MP11}.

\subsection{Classical Nodal Basis}
Following much of the notation in \cite{MP11}, suppose that we have a coarse quasi-uniform  discretization ${\cal T}_{H}$ of the unperforated domain $\Omega$ with mesh size $H$. 
We denote the interior nodes not on the boundary of the coarse mesh as ${\cal N}_{H}$.
Let the classical conforming $\mathbb{P}_{1}$ finite element space over ${\cal T}_{H}$ be given by $S_{H}$, 
and let $V_{H}=S_{H}\cap H^1_{0}(\Omega)$. We denote the nodal basis functions $\lambda_y$, that is for an interior node $y\in{\cal N}_{H}$, we have 
\begin{align}\label{P1}
 \lambda_{x}(x)=1 \text{ and } \lambda_{y}(x)=0, y\neq x.
\end{align}
This is a basis for  $V_{H}$. 
Let ${u}_{H}\in {V}_H$ be the function satisfying
\begin{align*}
 \int_{\Omega}\nabla {u}_{H} \nabla v dz=\int_{\Omega} g v dz, \text{ for all } v\in {V}_H .
\end{align*}

To move to the perforated domain it is useful to have some more notation. 
We denote the restriction operator of a function on $\Omega$ to $\perf$ by $\restrict:H^1_{0}(\Omega) \to H^1_{0}(\perf)$. We denote the space of finite   element functions \eqref{P1} restricted to the perforated domain as
$$\widetilde{V}_H=\{ w | \text{ there exists } u\in V_{H}, w=\restrict u\}={\cal R}\wt{V}_{H}.$$ From here we may define a coarse-grid variational form of \eqref{varform}. Indeed, let $\tilde{u}_{H}\in \widetilde{V}_H$ be the function satisfying
\begin{align}\label{varformcoarsegrid}
 \int_{\perf}\nabla \tilde{u}_{H} \nabla v dz=\int_{\perf} g v dz, \text{ for all } v\in \widetilde{V}_H.
\end{align}
However,  $\tilde{u}_{H}$ will not be a good approximation to $\tilde{u}$ unless $H$ is sufficiently small to resolve the microstructure. 
%

\subsection{Projective Quasi-Interpolation}
In this section, we develop the theory for a quasi-interpolation operator that is also a projection.  This projective quasi-interpolation gives stability properties required for the localization theory without the use of an auxiliary  "closeness to projection" lemma used in the theory of Cl\'{e}ment quasi-interpolation theory c.f.\ Lemma 1 of \cite{Henning.Morgenstern.Peterseim:2014}.  This requires the construction of a function that satisfies certain interpolation properties and  derivative bounds. However, in the case of perforated domains such a construction can be quite tedious and an alternate approach is utilized here. 

We will construct a quasi-interpolation operator that is also projective and satisfies the requisite local stability properties. 
For non-perforated domains, this is a well known modification of the operator of Cl\'{e}ment \cite{OhlbergerNotes}.
We denote the local patch $\text{supp(}\lambda_x)=\omega_{x}$ for $x\in{\cal N}_{H}$ and, subsequently, the perforated patch as $\wt{\omega}_x=\omega_{x}\cap \perf$.
First, we define the local patch $L^2$ projection  ${\cal P}_{x}: L^2(\wt{\omega}_{x})\to \wt{V}_{H}(\wt{\omega}_{x}),$ as the operator such that for $u\in H^1_{D}(\wt{\omega}_{x})$ 
\begin{align}\label{L2proj}
\int_{\wt{\omega}_{x}}( {\cal P}_{x}u) v_{H} dz= \int_{\wt{\omega}_{x}}u v_{H} dz \text{ for all } v_{H}\in  \wt{V}_{H}(\wt{\omega}_{x}).
\end{align}
From this we define the interpolation operator $\Qintp: \Hperf \to \wt{V}_{H}$ for $u\in \Hperf$ as
\begin{align}\label{proj}
\Qintp u=\sum_{x\in{\cal N}_{H}} ({\cal P}_{x} u)(x){\cal R}  \lambda_{x}.
\end{align}

Given a function $\phi\in H^1_{D}(\perf)$ with support contained in a patch of triangles  $\wt{\omega}_{x}$, then, by the definition of the quasi-interpolation \eqref{proj}, it is clear that $\text{supp}(\Qintp(\phi))\not \subset\wt{\omega}_{x}$  in general, as the boundary nodes on the patch $\wt{\omega}_{x}$ will add a contribution smearing out the function. To deal with this issue we require some notation and definitions.
Using the definition and notation in \cite{Henning.Morgenstern.Peterseim:2014}, we define for any patch $\wt{\omega}_{x}$
 the extension patch 
\begin{subequations}\label{patches}
\begin{align}
 \wt{\omega}_{x}&=\wt{\omega}_{x,0}=\text{supp}( \lambda_{x}) \cap \perf,\\
\wt{\omega}_{x,k}&=\text{int}(\cup \{ T\in T_{H} | T\cap \wt{\omega}_{k-1}\neq \emptyset \}\cap \perf,
\end{align}
\end{subequations}
for $k=1,2,3,4,\dots$.
With this notation we have $\text{supp}(\Qintp(\phi)) \subset \overline{\wt{\omega}}_{x,1}$ if $\text{supp}(\phi) \subset\wt{\omega}_{x,0}$ for the interpolator \eqref{proj}.

%
%
%
We have the following stability and local approximation of the quasi-interpolation operator $\Qintp$ defined by \eqref{proj}, along with the desired projective properties.

\begin{lemma}
  There exists a constant $\local>0$, for all $u\in \Hperf$, such that 
   \begin{align}\label{stableproj}
 &H^{-1}\norm{u-\Qintp u}_{L^2(\wt{\omega}_{x,0} ) }+\norm{\nabla (u-\Qintp u)}_{L^2(\wt{\omega}_{x,0})} \leq  \local  \norm{\nabla  u}_{L^2(\wt{\omega}_{x,1} )},
\end{align}
where $\local=C C_{P}$.  Here $C_{P}$ is the Poincar\'{e} constant in perforated domains. Here, $C$ is a benign  constant not depending on $H$ or $\eta$. Moreover, the interpolation $\Qintp$ is a projection. 
\end{lemma}
\begin{proof}
See  Appendix \ref{quasiproof}.
\end{proof}

It is important to note  that we have a Poincar\'{e} constant in the above estimate. Since our domain can have complicated microstructure we must be careful when analyzing estimates that contain this constant. We suppose that we have the following general Poincar\'{e} inequality for each patch $\wt{\omega}_{x}$ for $x\in {\cal N}_{H}$. Moreover, we shall suppose that this constant serves as a global bound with respect to $H$ and $\eta$. The analysis of such a constant will be considered in Section \ref{poinsection}.

For all $\wt{\omega}_{x}$ with $x\in {\cal N}_{H}$, we have for $\phi\in H^1(\perf)$
\begin{align}\label{poincarelocal}
 \norm{\phi-\bar{\phi}}_{L^2(\wt{\omega}_{x})}\leq H C_P^{x}(\eta,H) \norm{\nabla \phi}_{L^2(\wt{\omega}_{x})},
\end{align}
Where $C^x_P(\eta,H)$ may depend on the diameter of the triangulation, and subsequently ${\omega}_{x}$, and  its characteristic microstructure parameter $\eta$. We denote $C_{P}(\eta,H)=\max_{x\in {\cal N}_{H}}C_{P}^x(\eta,H)$ and will drop the notation $(\eta,H)$ in many of the auxiliary  estimates in the  Appendix   \ref{auxsection} and throughout the paper when there is no ambiguity.

\section{Multiscale Splitting and Basis}
We now will construct our multiscale approximation space to handle the oscillations created by the perforated microstructure. The main ideas of this splitting can be found in \cite{Henning.Morgenstern.Peterseim:2014,MP11} and references therein. As noted before the coarse mesh space restricted to $\perf$ can not resolve the features of the microstructure and these fine-scale features must be captured in the multiscale basis. We begin by constructing fine-scale spaces.

We define the kernel of the perforated interpolation operator to be 
\begin{align*}
 \widetilde{V}^f=\{v\in \Hperf \;| \; \Qintp v=0\},
\end{align*}
where $\Qintp$ is defined by \eqref{proj}. This space will represent the small scale features not captured by $\wt{V}_H$. We define the fine-scale projection $Q_{\perf}  :\wt{V}_H\to \wt{V}^f$ to be the operator such that for $v\in \wt{V}_H$ we compute $Q_{\perf} (v)\in \wt{V}^f $ as
\begin{align}\label{corrector}
 \int_{\perf}\nabla Q_{\perf} (v) \nabla w dz=\int_{\perf}\nabla v \nabla w dz, \text{ for all } w\in \wt{V}^f.
\end{align}
This projection gives an orthogonal splitting $\Hperf=\wt{V}^{ms}_{H}\oplus \wt{V}^f$ with $\wt{V}^{ms}_{H}=(\wt{V}_{H}-Q_{\perf} ( \wt{V}_{H})).$ We can decompose any $u\in \Hperf$ as $u=u^{ms}+u^f$ with $\int_{\perf} \nabla u^{ms} \nabla u^f dz=0$.
This modified coarse space is referred to as the multiscale space and contains fine-scale geometric information. The multiscale Galerkin approximation  $u_{H}^{ms}\in \wt{V}_{H}^{ms}$ satisfies 
\begin{align}\label{msvarform}
 \int_{\perf}\nabla u_{H}^{ms} \nabla v dz=\int_{\perf}gv dz \text{ for all } w\in \wt{V}^{ms}_{H} .
\end{align}

To construct the basis for the multiscale space $\wt{V}_{H}^{ms}$ we construct an adapted coarse grid basis.   We define the corrector $\phi_x=Q_{\perf} (\lambda_x )$ to be the solution to 
\begin{align}\label{varformcell}
 \int_{\perf}\nabla  \phi_x \nabla w dz =\int_{\perf} \nabla \lambda_x  \nabla w dz, \text{ for all } w\in \widetilde{V}^f.
\end{align}
We then define the perforated multiscale space $\widetilde{V}^{ms}_H$  to be the functions spanned by
\begin{align}
 \widetilde{V}^{ms}_H=\text{span}\{\restrict \lambda_x-\phi_x|x\in {\cal N}_{H} \}.
\end{align}
Note that the corrector problem \eqref{corrector} is posed on the global domain. Thus, the corrections will have  global support and as such have limited practical use. However, in the following analysis we show that the basis can be localized. 

The key issue with constructing the solution to \eqref{msvarform} is the calculation of the corrector on a global basis.  However, it can be shown that the corrector decays exponentially fast.  To this end, we define the localized fine-scale space to be the fine-scale space extended by zero outside the patch, that is $\wt{V}^f(\wt{\omega}_{x,k})=\{v\in \wt{V}^f |\text{   } v|_{\perf\backslash \wt{\omega}_{x,k}}=0\}$. 
It is convenient to introduce some notion here similar to that introduced in \cite{Henning.Morgenstern.Peterseim:2014}. We let  for some $x\in {\cal N}_{H}$ and $k\in \mathbb{N}$ the local corrector operator $Q_{x,k}: \wt{V}_{H}\to\wt{V}^f(\wt{\omega}_{x,k}),$ be defined such that given a $u_{H}\in \wt{V}_{H}$
\begin{align}\label{Qcorrector}
 \int_{\wt{\omega}_{x,k}}\nabla Q_{x,k}(u_{H}) \nabla w dz=\int_{\wt{\omega}_{x}}\hat{\lambda}_{x}\nabla u_{H} \nabla w dz, \text{ for all } w \in \wt{V}^f(\wt{\omega}_{x,k}),
\end{align}
where $\hat{\lambda}_{x}=\frac{{\lambda}_{x}}{\sum_{y}\lambda_{y}}$ is  augmented so that the collection $\{ \hat{\lambda}_{x}\}_{x\in {\cal N}_{H}}$ is a partition of unity. This is done because the Dirchlet condition makes the standard basis not a partition of unity near the boundary.
For a practical evaluation of $Q_{x,k}$,  we may precompute for any neighbor $y\in {\cal N}_{H}\cap \wt{\omega}_{x}$ of $x$ the following
\begin{align}\label{Qcorrectorbasis}
 \int_{\wt{\omega}_{x,k}}\nabla Q_{x,k}(\lambda_y) \nabla w dz=\int_{\wt{\omega}_{x}}\hat{\lambda}_{x}\nabla \lambda_y \nabla w dz, \text{ for all } w \in \wt{V}^f(\wt{\omega}_{x,k}).
\end{align}
We then write $Q_{x,k}(u_{H})=\sum_{y\in{\cal N}_{H}\cap \wt{\omega}_{x}}u_{H}(y)Q_{x,k}(\lambda_y)$ and so must only compute over small number of nearby nodes for each $x$.
Moreover, we are able to exploit local periodic structures due to the fact that a drastically reduced number of corrector problems must be computed, assuming the coarse-grid is chosen properly.

We denote the global corrector operator as $$Q_{k}(u_{H})=\sum_{x\in {\cal N}_{H}}Q_{x,k}(u_{H}).$$  With this  notation, we   write the truncated multiscale space as
 $$ \widetilde{V}^{ms}_{H,k}=\text{span}\{u_{H}-Q_{k}(u_{H}) | u_{H}\in \wt{V}_{H}  \}.$$ 
 Moreover, note also that for sufficiently large $k$, we recover the full domain and obtain the ideal corrector with functions of global support, denoted $Q_{\tilde{\Omega}}.$
%
%
The corresponding multiscale approximation to \eqref{varform} is 
\begin{align}\label{localmsvarform}
 \int_{\perf}\nabla u_{H,k}^{ms} \nabla v dz=\int_{\perf}g v dz \text{ for all } w\in \wt{V}^{ms}_{H,k}.
\end{align}
%
%
\section{Error Analysis}

In this section we present the error introduced by using  \eqref{msvarform} on the global domain to compute the solution to \eqref{varform}. Then, we show how localization effects the error when we use  \eqref{localmsvarform} on truncated domains to compute the same solution. Meanwhile, we must carefully account for the effects of the Poincar\'{e} constant from \eqref{poincarelocal} in the estimate as in certain domains this may depend on the microstructure or coarse grid diameters. 

\subsection{Error with Global Support}

\begin{theorem}\label{errorglobal}
Suppose that $u\in \Hperf$ satisfies \eqref{varform} and that $u_{H}^{ms}\in \wt{V}_{H}^{ms}$ , with correctors of global support in \eqref{varformcell}, satisfies \eqref{msvarform}. Then, we have the following error estimate
\begin{align}\label{eq:errorglobal}
\TwoNorm{\nabla u-\nabla u_{H}^{ms}}{\perf}\leq C_{ol}^{\frac{1}{2}} \local \TwoNorm{H g}{\perf}.
\end{align}
\end{theorem}
\begin{proof}
Again we use the local stability property of $\Qintp$ the local interpolation operator in \eqref{stableproj}. From the orthogonal splitting of the spaces it is clear that $ u- u_{H}^{ms}=u^f \in \wt{V}^f$ and $\Qintp(u^f)=0$. Thus, using the stability inequality we have 
\begin{align*}
&\TwoNorm{\nabla u-\nabla u_{H}^{ms}}{\perf}^2=\TwoNorm{\nabla u^f}{\perf}^2=\int_{\perf}g\left( u^f-\Qintp(u^f)\right)dz\\
&\leq \sum_{x\in{\cal N}_{H}} \TwoNorm{g}{\perf}\local H_{} \TwoNorm{u^f}{\perf}\leq \frac{\local^2}{2 \eps }  \TwoNorm{H g}{\perf}^2+\frac{\eps}{2}\sum_{x\in{\cal N}_{H}}  \TwoNorm{\nabla u^f}{\wt{\omega}_{x}}^2.
\end{align*}
Let $C_{ol}$ be the maximal number of elements covered by a patch $\wt{\omega}_{x}$ and we suppose the mesh is so that this is uniformly bounded.  Taking $\eps=C^{-1}_{ol}$ we arrive at the estimate \eqref{eq:errorglobal}.
\qed
\end{proof}

\subsection{Error with Localization}

In this  section  we show the error due to truncation with respect to patch extensions. The key lemma needed is the following estimate, the proof can be found in  Appendix \ref{auxsection}.

\begin{lemma}
Let $u_{H}\in \wt{V}_{H}$, let $Q_{m}$ be constructed from \eqref{Qcorrector}, and $Q_{\perf}$ defined to be the "ideal" corrector without truncation, then 
\begin{align}
\TwoNorm{\nabla( Q_{\perf}(u_{H})-Q_{m}(u_{H}))  }{\perf}\leq   m^{\frac{d}{2}}C_{4} \theta^{m  }   \TwoNorm{ \nabla Q_{\perf}(u_{H})  }{\perf},
\end{align}
with  $\theta\in (0,1)$, $C_{4}=C_{3}    (1+ C_{1}^2 )^{\frac{1}{2}},$ and $C_{3}=(1+C_{1}+C\local)$.
\end{lemma}

The lemma gives the decay in the error as the  truncated corrector approaches the ideal corrector of global support. With this lemma we are able to state and prove Theorem \ref{errorlocal}.
\begin{theorem}\label{errorlocal}
Suppose that $u\in \Hperf$ satisfies \eqref{varform} and that $u_{H,m}^{ms}\in \wt{V}_{H,m}^{ms}$, with local correctors  calculated from \eqref{Qcorrector}, satisfies \eqref{localmsvarform}. Then, we have the following error estimate
\begin{align}\label{eq:errorglobal}
\TwoNorm{\nabla u-\nabla u_{H,m}^{ms}}{\perf}\leq \left(C_{ol}^{\frac{1}{2}} \local H+ m^{\frac{d}{2}}C_{5} \theta^{m  }\right) \TwoNorm{g}{\perf},
\end{align}
with $\theta \in (0,1)$ a constant depending on Poincar\'{e} constants. In addition, with respect to Poincar\'{e} constants we have $$C_{ol}^{\frac{1}{2}} \local\leq C C_{P}, \text{ and } C_{5}\leq C C^4_{P}.$$
$C$ being benign constants not depending on $H$ or $\eta$.
\end{theorem}
\begin{remark}
Note from Lemma \ref{decaylemma}, we  have $\theta=e^{-\frac{1}{\lceil C_{2} e \rceil+2}}\in (0,1)$,  here  $C_{2}=(C_{1} +C\local )\approx C_{P}^{\frac{3}{2}}$. Thus, the Poincar\'{e} constant effects the estimate in Lemma \ref{decaylemma}  insofar as it may slower the decay rate of the exponential and not lead to some sort of exponential "blow-up" with respect to patch extensions. 
\end{remark}
\begin{proof}[Proof of Theorem \ref{errorlocal}]
We let $u^{ms}_{H}=u_{H}-Q_{\perf}(u_{H})$ be the ideal global multiscale solution satisfying \eqref{msvarform}, we have using Theorem \ref{errorglobal} and Lemma 
\ref{localglobal}
\begin{align*}
\TwoNorm{\nabla u-\nabla u_{H,m}^{ms}}{\perf}&\leq\TwoNorm{\nabla u-\nabla u_{H}^{ms}}{\perf}+\TwoNorm{\nabla u_{H}^{ms}-\nabla u_{H,m}^{ms}}{\perf}\\
&\leq C_{ol}^{\frac{1}{2}} \local H\TwoNorm{ g}{\perf}+\TwoNorm{\nabla( Q_{\perf}(u_{H})-Q_{m}(u_{H}))  }{\perf}\\
&\leq C_{ol}^{\frac{1}{2}} \local H\TwoNorm{ g}{\perf}+ m^{\frac{d}{2}}C_{4} \theta^{m   } \TwoNorm{ \nabla Q_{\perf}(u_{H})  }{\perf}. 
\end{align*}
Finally, noting that  from \eqref{Qcorrector} and  \eqref{msvarform} we have 
\begin{align*}
\TwoNorm{ \nabla Q_{\perf}(u_{H})  }{\perf} \leq \TwoNorm{ \nabla u^{ms}_{H} }{\perf}\leq C_{P}  \TwoNorm{ g }{\perf},
\end{align*}
applying this above we obtain the required estimate.

To obtain the relationship on the above estimate to the Poincar\'{e} constants note from \eqref{stableproj} that $\local \approx C_{P}$. From Lemma \ref{qi}, we have 
$C^2_{1}=C_{lip}^2\local +\local^3$. We have $C_{lip}\leq C C_{P}$, thus $C_{1}\approx C_{P}^\frac{3}{2}$. From Lemma \ref{localglobal}, 
$$C_{4}=C_{3}(1+C_{1}^2)^{\frac{1}{2}}=(1+C_{1}+\local)(1+C_{1}^2)^{\frac{1}{2}}\approx C_{P}^3,$$
and so $C_{5}=C_{P}C_{4}\approx C_{P}^4$
\qed
\end{proof}


\section{Estimates for  Poincar\'{e} Inequalities}\label{poinsection}
In this section, we discuss the tools required to estimate the constant $C_{P}$ in certain physically interesting cases. The following techniques were developed and used effectively in the context of weighted Poincar\'{e} inequalities in the setting of contrast dependence \cite{Pechstein01042013} and references therein. We follow much of the notation presented in that work, however, here we adapt the techniques to complex domain geometries and not contrast independent estimates. 
The case of high-contrast will be discussed in the forthcoming preprint \cite{PS2014}.
We begin by building the necessary framework to effectively estimate $C_P$ in a constructive way. Throughout this section we shall suppose that $H>\eta$, the characteristic separation and length scale size. 

We begin by fixing $x\in {\cal N}_{H}$ and examining a single patch $\wt{\omega}_{x}$. We will have a slight abuse of notation we call this constant Poincar\'{e} $C_{P}$ as we will take a maximum over all patches. We suppose that the estimate on this  patch bounds all the others. 
We begin as in \cite{Pechstein01042013}, let ${\cal Y}=\{Y_{l}\}_{l=1}^{n}$ be a non overlapping partitioning of $\wt{\omega}_{x}$ into open, connected Lipschitz  polytopes so that
\begin{align*}
\overline{\wt{\omega}}_{x}=\bigcup_{l=1}^{n} \overline{Y}_{l},
\end{align*}
with $H=\text{diam}(\wt{\omega}_{x}).$  For $u\in H^1(\wt{\omega}_{x})$ and $(d-1)$ dimensional manifold $X\subset \wt{\omega}_{x}$ we define the average 
\begin{align*}
\bar{u}^X=\frac{1}{|X|}\int_{X}u ds,
\end{align*}
here the above integral is taken with respect to the $(d-1)$ dimensional real Lebesgue measure $ds$. 

We call a region $P_{l_{1},l_{s}}=( \overline{Y}_{l_{1}}\cup  \overline{Y}_{l_{2}}\cup\cdots \cup  \overline{Y}_{l_{s}})$ a path if  for each $i=1,\dots,s-1$, the regions  $\overline{Y}_{l_{i}}$ and $\overline{Y}_{l_{i+1}}$ share a common $(d-1)$-dimensional manifold. Here, $s$ is the length of the path $P_{l_{1},l_{s}}$. Suppose there is a path $P_{k,l^*}$ from $Y_{k}$ to $Y_{l^*}$ with path length $s_{k}$. 
Let $X^*\subset \bar{Y}_{l^*}$ be a $(d-1)$ dimensional manifold, then for each $k=1,2,\dots, n$ let $c_{k}^{X^*}>0$ be the best constant so that 
\begin{align}\label{poinpath}
\TwoNorm{u-\bar{u}^{X^*}}{Y_k}^2\leq (c_{k}^{X^*})^2 H^2\TwoNorm{\nabla u}{P_{k,l^*}}^2,
\end{align}
for all $u\in H^1(P_{k,l^*})$. Note here we make a change of notation compared to \cite{Pechstein01042013}, in that we replace $c_{k}^{X^*}$ with its square, similarly with $C_{P}$ and related constants.

We now define the Poincar\'{e} inequality for a single domain, most likely in our application to be a simplicial domain such as a triangle, tetrahedron, or perhaps nonsimplicial, but regular,  such as quadrilaterals,  parallelepiped, or curved elements. The key here being that each simplex has a trivially bounded Poincar\'{e} constant.
For any Lipschitz domain $Y \subset \mathbb{R}^d$ and for any $(d-1)$ dimensional manifold $X\subset \bar{Y}$, we denote $C_{P}(Y;X)>0$ to be the best constant such that 
\begin{align}\label{poinsimplex}
\TwoNorm{u-\bar{u}^{X}}{Y}^2\leq C^2_{P}(Y;X) \text{diam}(Y)^2\TwoNorm{\nabla u}{Y}^2, 
\end{align}
for all $u\in H^1(Y)$. We have the following lemma relating the constants in \eqref{poinpath} and \eqref{poinsimplex}. 

\begin{lemma}\label{lemmacx}
Suppose $P_{k,l^*}$ is a path as defined above of length $s$ with $l_{1}=k$ and $l_{s}=l^*$. We let $X_{0}=X_{1}$ and $X_{s}=X^*$. Then, the constant from \eqref{poinpath} can be bounded by the constants related to inequality \eqref{poinsimplex}
\end{lemma}
\begin{align}
(c_{k}^{X^*})^2\leq 4 \sum_{i=1}^s \frac{|Y_k|}{|Y_{l_{i}}|} \frac{\text{diam}(Y_{l_i})^2}{H^2}\mbox{max}(C^2_P(Y_{l_i},X_{i-1}),C^2_{P}(Y_{l_i},X_{i}))
\end{align}
\begin{proof}
By using the standard telescoping argument
\begin{align*}
\TwoNorm{u-\bar{u}^{X^*}}{Y_k}& \leq \TwoNorm{u-\bar{u}^{X_1}}{Y_k}+\sum_{i=2}^s \sqrt{|Y_{k}|}|\bar{u}^{X_{i-1}}-\bar{u}^{X_{i}}|,
\end{align*}
 and the use of \eqref{poinsimplex} we have 
 \begin{align*}
 \TwoNorm{u-\bar{u}^{X_1}}{Y_k}^2 \leq C^2_{P}(Y_k;X_1) \text{diam}(Y_k)^2\TwoNorm{\nabla u}{Y_k}^2.
 \end{align*}
Fixing $i$ we have for the second term 
\begin{align*}
|\bar{u}^{X_{i-1}}-\bar{u}^{X_{i}}|^2&\leq \frac{2}{|Y_{l_i}|}\left(  \TwoNorm{u-\bar{u}^{X_{i-1}}}{Y_{l_i}}^2+\TwoNorm{u-\bar{u}^{X_{i}}}{Y_{l_i}}^2 \right)\\
&\leq  \frac{2}{|Y_{l_i}|}  \left( (C^2_{P}(Y_{l_i};X_{i-1})+C^2_{P}(Y_{l_i};X_{i})) \text{diam}(Y_{l_i})^2\TwoNorm{\nabla u}{Y_{l_i}}^2\right)\\
&\leq 4 \text{max}(C^2_{P}(Y_{l_i};X_{i-1}),C^2_{P}(Y_{l_i};X_{i})) \frac{ \text{diam}(Y_{l_i})^2}{|Y_{l_i}|}   \TwoNorm{\nabla u}{Y_{l_i}}^2.
\end{align*}
A final application of the Cauchy inequality yields the desired result. 
\qed
\end{proof}

We define $(C_{P})^2=\sum_{k=1}^n (c_{k}^{X^*})^2$ and we have the general full Poincar\'{e} inequality
\begin{align}\label{pointotal}
\TwoNorm{u-\avrg{u}{\wt{\omega}_{x}}}{\wt{\omega}_{x}}^2\leq\TwoNorm{u-\bar{u}^{X^*}}{\wt{\omega}_{x}}^2\leq C_{P}^{2} H^2\TwoNorm{\nabla u}{\wt{\omega}_{x}}^2,
\end{align}
recall here  $\avrg{u}{\wt{\omega}_{x}}=\frac{1}{|\wt{\omega}_x|}\int_{\wt{\omega}_x}u dz$ is the optimal minimizing constant.
\subsection{Poincar\'{e} Inequalities with Geometric Parameters}

To obtain better bounds on $C_{P}$ we must in turn obtain a systematic way to obtain bounds on $c_{k}^{X^*}$. To this end, we will use the following two technical lemmas.
The first of which estimates the constant $C_{P}(K;F)$ for a simplex.
\begin{lemma}\label{simplicial} Let $K$ be a simplex (or parallelepiped), and $F$ one of its faces, then
\begin{align*}
C^2_{P}(K;F)\leq \frac{7}{5}.
\end{align*}
\end{lemma}
\begin{proof}
See Appendix of \cite{Pechstein01042013}.
\end{proof}
We also state a common estimate for regular triangulation.
\begin{lemma}\label{shapeReg} Let $K$ a nondegenerate simplex and $\rho(K)$ the diameter of the largest sphere inscribed in $\bar{K}$, then  
\end{lemma}
\begin{align*}
|K|\geq \text{diam}(K)\left(\frac{\rho(K)}{2} \right)^{d-1}.
\end{align*}
\begin{proof}
See Appendix of \cite{Pechstein01042013}.
\end{proof}
Let ${\cal Y}=\{Y_{l}\}_{l=1}^{n}$ be a conforming simplicial triangulation  of $\wt{\omega}_{x}$ and we define the geometric parameters for $l=0,\dots, n$, 
$\eta_{l}=\text{diam}(Y_{l})$, $\eta=\text{max}(\eta_{l})$, and $\eta_{min}=\text{min}(\eta_{l})$. We define the shape-regularity constant 
\begin{align*}
C^{\cal Y}_{reg}=\max_{l=1}^{n}\left(\frac{\text{diam}(Y_{l})}{\rho(Y_{l})}\right).
\end{align*}
We call a partition of $\wt{\omega}_{x}$ shape regular if there is a uniform bound for $C^{\cal Y}_{reg}$ and quasi-uniform if in addition to shape regular we have $\eta/\eta_{min}$ uniformly bounded. 
With this type of a partition we are able to obtain a useful tool to estimate $C_{P}$.
\begin{lemma}\label{meth1} Let ${\cal Y}=\{Y_{l}\}_{l=1}^{n}$ be a shape regular  simplicial partition  of $\wt{\omega}_{x}$, with $X^*$ a facet of $Y_{l^*}$.  We  denote the path length for $Y_{k}$ to $Y_{l^*}$ by $s_{k}$.
Then, we have the bound
\end{lemma}
\begin{align}\label{poinest1}
C_{P}^2 \leq  \left(\frac{28}{5}\right) 2^{d+1}(C^{\cal Y}_{reg})^{d-1} \sum_{k=1}^{n} \frac{s_k |Y_{k}|}{H^2 \eta_{min}^{d-2}}.
\end{align}
\begin{proof}
From Lemma \ref{lemmacx} we have for a fixed $k$ and path $P_{k,l^*}$ of simplicial domains that
\begin{align*}
(c_{k}^{X^*})^2\leq 4 \sum_{i=1}^{s_k} \frac{|Y_k|}{|Y_{l_{i}}|} \frac{\text{diam}(Y_{l_i})^2}{H^2}\mbox{max}(C^2_P(Y_{l_i},X_{i-1}),C^2_{P}(Y_{l_i},X_{i})).
\end{align*}
For $i=1,\dots s_k$, using Lemma \ref{simplicial} we see, taking $K=Y_{l_{i}}$ and $F=X_{i-1}$ or $F=X_{i}$,  that  $$\mbox{max}(C^2_P(Y_{l_i},X_{i-1}),C^2_{P}(Y_{l_i},X_{i}))\leq \frac{7}{5}.$$
From shape regularity and Lemma \ref{shapeReg} we have 
$$
\frac{\text{diam}(Y_{l_i})^2}{|Y_{l_{i}}|}\leq 2^{d+1}(C^{\cal Y}_{reg})^{d-1} \eta_{l_{i}}^{2-d}.
$$
Taking the minimum $\eta_{l_{i}}$ we have 
\begin{align*}
(c_{k}^{X^*})^2\leq  \left(\frac{28}{5}\right)\sum_{i=1}^{s_k} 2^{d+1}(C^{\cal Y}_{reg})^{d-1} \frac{|Y_k|}{H^2  \eta_{min}^{d-2}}\leq  \left(\frac{28}{5}\right)  2^{d+1}(C^{\cal Y}_{reg})^{d-1} \frac{{s_k}|Y_k|}{H^2  \eta_{min}^{d-2}}.
\end{align*}
Summing from $k=1,\dots n$ over the simplices we obtain the estimate.
\qed
\end{proof}
As noted in \cite{Pechstein01042013}, the above lemma can give "worst case" scenarios for estimates on Poincar\'{e} constants. To illustrate the usefulness of the above estimate \eqref{poinest1} to obtain rough bounds we give an illustrative example. It can be easily seen that the estimate  \eqref{poinest1} will grow when the path lengths, $s_k$, are large. This can be especially bad in highly tortuous microstructures. We illuminate this by considering a two-dimensional filamented microstructure. 

Suppose we take our domain to be ${\omega}_{x}=[0,H]^2$, and inside we have the solid microstructure, $S_\eta$, given by thin filamented structures. More precisely, 
$$S_\eta=\bigcup_{j=0}^{N_{\eta}}\left( ([4\eta j,4\eta j+\eta]\times [0,H-\eta])\cup ([4\eta j+2\eta,4\eta j+3\eta]\times [\eta, H]) \right),$$
where $N_{\eta}\leq \left \lfloor{ \frac{H}{4 \eta}}\right \rfloor$.  Note we take here the floor of  $\frac{H}{4 \eta}$ to ensure $N_\eta$ is such that we have the right hand side boundary free of microstructure intersections. This is done since we will suppose that $X^*=\{H\} \times[0,H]$ and we wish this boundary to be a part of the domain $\wt{\omega}_{x}$ defined as 
$$\wt{\omega}_{x}={\omega}_{x}\backslash S_{\eta}.$$

Suppose we have a uniform shape regular triangularization of $\wt{T}$ denoted again by ${\cal Y}=\{Y_{l} \}_{l=1}^{n}$. Moreover, we suppose that $|Y_{k}|\approx \eta^2$, for all $k=1,\dots n$. We denote $s_{max}$ the maximal path length from $Y_k$ to $X^*$. Then, \eqref{poinest1} becomes 
\begin{align}
C_{P}^2 \leq  C C^{\cal Y}_{reg}\frac{n s_{max} \eta^2 }{H^2 },
\end{align}
here $C$ is a benign constant. To estimate $s_{max}$, we take the simplex farthest from the right hand side boundary $X^*$ denoted $Y_{1}$ to construct the longest path. Note that  $Y_{1}$ is formed by bisecting $[0,\eta]\times [H-\eta,H]$ into two equal triangles and taking the one adjacent to the left boundary of $[0,H]^2$. 
We can see that in each filament the path length is $O(\frac{H}{\eta})$ and there are $O(N_{\eta})\approx O(\frac{H}{\eta})$ filaments. Hence $s_{max}\approx O((\frac{H}{\eta})^2)$, and in addition, we see also that $n\approx O((\frac{H}{\eta})^2)$ as this is the number of triangles in the partition of $\wt{\omega}_{x}$. We thus obtain the estimate for the Poincar\'{e} constant 
\begin{align}\label{poinest2}
C_{P}^2 \leq  C C^{\cal Y}_{reg}\left(\frac{ H }{\eta }\right)^2.
\end{align}
Taking the maximum over the possible constants over the patches, and applying this estimate to Theorem \ref{errorlocal} we have 
\begin{align}
\TwoNorm{\nabla u-\nabla u_{H,m}^{ms}}{\perf}\leq C\left(\left(\frac{ H }{\eta }\right)H+ m^{\frac{d}{2}}\left(\frac{ H }{\eta }\right)^4 e^{-\left(\frac{\eta}{H}\right)^{\frac{3}{2}}m  } \right) \TwoNorm{g}{\perf}.
\end{align}
%
Thus, it is possible to see how the closeness of the microstructure could theoretically effect the convergence estimate via the Poincar\'{e}.
The constant in the exponential may also effect the decay with respect to patch extension. 
 However, the above example is meant to represent a very poor scenario. 

\begin{remark}
It is important to note here that the above estimate holds for $H>\eta$. 
 The Poincar\'{e} constants in the other regimes would certainly not be expected to yield a convergence order of more than $H$ in a regime such as $H<\eta$. In this regime, where scale separation is not the case, the notation $C_{P}=\text{max}(O(1),\frac{H}{\eta})$ would be more appropriate.
\end{remark}

\subsection{Poincar\'{e} Constants for Isolated Perforations}

In the previous section, we presented a general method for determining the dependence of the Poincar\'{e} constant on the microstructure. This estimate offers a sort of worst case scenario for such a constant. In this section, we will show that for isolated convex particles in two-dimensions fairs much better and in fact can be shown to be uniformly bounded. 
To this end we will need some further results again drawn from the work of \cite{Pechstein01042013}.

 \begin{lemma}\label{integralest} Let ${\cal Y}=\{Y_{l} \}_{l=1}^{n}$ be a shape regular and quasi-uniform  simplicial partition  of $\wt{\omega}_{x}$ with mesh size $\eta>0$. Further, let $X^*=\cup_{j=1}^J F_{j}$  
 , where $F_{j}$ are the faces of $Y_j$ and for simplicity suppose $X^*$ is not perforated. For $k\in {\cal I}=\{1,\dots, n\}$ and $j\in {\cal J}=\{1,\dots J\}$ we denote $P_{k,j}$ the path from $Y_k$ to $Y_j$.
Then, we have the bound
\begin{align}\label{integralesteq}
\int_{F_j}\int_{Y_{k}}|u(x)-u(y)|^2 dy ds_x\leq C s_{k,j} \eta^{d+1}\TwoNorm{\nabla u}{P_{k,j}}^2,
\end{align}
for all $u\in H^1(P_{k,j})$ and where $s_{k,j}$ is the path length of $P_{k,j}$.
\end{lemma}
\begin{proof} We proceed as in \cite{Pechstein01042013},
note that 
\begin{align*}
\int_{F_j}\int_{Y_{k}}|u(x)-u(y)|^2 dy ds_x& \leq\int_{F_j}\int_{Y_{k}}|u(x)-\bar{u}^{F_j}|^2 dy ds_x+\int_{F_j}\int_{Y_{k}}|\bar{u}^{F_{j}}-u(y)|^2 dy ds_x\\
&\leq |F_{j}|\TwoNorm{u-\bar{u}^{F_j}}{Y_k}^2+|Y_{k}|\TwoNorm{u-\bar{u}^{F_j}}{F_j}^2.
\end{align*}
Using Lemma \ref{lemmacx} with $X^*=F_{j}$ and shape regularity we have 
$$
\TwoNorm{u-\bar{u}^{F_j}}{Y_k}^2\leq C s_{k,j}\frac{|Y_{k}|}{\eta^{d-2}} \TwoNorm{\nabla u}{P_{k,j}}^2,
$$
and by a transformation argument we have 
$$
\TwoNorm{u-\bar{u}^{F_j}}{F_j}^2\leq C \eta \TwoNorm{\nabla u}{Y_{j}}^2,
$$
Thus, we have 
\begin{align*}
\int_{F_j}\int_{Y_{k}}|u(x)-u(y)|^2 dy ds_x&\leq |F_{j}|\TwoNorm{u-\bar{u}^{F_j}}{Y_k}^2+|Y_{k}|\TwoNorm{u-\bar{u}^{F_j}}{F_j}^2\\
&\leq  C s_{k,j}\frac{ |F_{j}|  |Y_{k}|}{\eta^{d-2}} \TwoNorm{\nabla u}{P_{k,j}}^2+C|Y_{k}| \eta \TwoNorm{\nabla u}{Y_{j}}^2\\
&\leq  C (s_{k,j}\frac{\eta^{2d-1} }{\eta^{d-2}}+ \eta^{d+1}) \TwoNorm{\nabla u}{P_{k,j}}^2\leq C s_{k,j} \eta^{d+1}\TwoNorm{\nabla u}{P_{k,j}}^2,
\end{align*}
here we used that $|Y_k|\approx \eta^d$ and $|F_{j}|\approx \eta^{d-1}$.
\qed
\end{proof}
Now we are able to obtain an alternative estimate approach compared to Lemma \ref{meth1}.
\begin{lemma}\label{lemmaest2}
Assuming the assumptions of Lemma \ref{integralest}, we let $X^*=\cup_{j=1}^J F_{j}$ and path $P_{k,j}$. Then, we have the estimate 
\begin{align}
C_{P}^2\leq C \frac{s_{max} r_{max} \eta^{d+1}}{|X^*|H^2},
\end{align}
here $s_{max}=\max(s_{k,j})$ and 
$$
r_{max}=\max_{i\in {\cal I}}  |\{(k,j)\in {\cal I}\times {\cal J}: Y_{i}\subset P_{k,j}\}   |,
$$
is the maximal number of times any  simplices $Y_i$ is in a path.
\end{lemma}
\begin{proof}
Without loss of generality we suppose for $u\in H^1(\wt{\omega}_{x})$ that $\bar{u}^{X^*}=0$. Using the identity, $(u(z)-u(y))^2=u(z)^2-2u(z)u(y)+u(y)^2$   for $z \in X^*$ (note here $z$ is a $d-1$ dimensional variable) and integrating we have 
$$
\int_{X^*}\int_{\wt{\omega}_{x}}(u(z)-u(y))^2dyds_z=|\wt{\omega}_{x}|\TwoNorm{u}{X^*}^2-2\int_{X^*}\int_{\wt{\omega}_{x}}u(z)u(y)dyds_z+|X^*|\TwoNorm{u}{\wt{\omega}_{x}}^2.
$$
The middle term vanishes and we have 
$$
|X^*|\TwoNorm{u}{\wt{\omega}_{x}}^2\leq \int_{X^*}\int_{\wt{\omega}_{x}}(u(z)-u(y))^2dyds_z=\sum_{k\in {\cal I}}\sum_{j\in {\cal J}}\int_{F_{j}}\int_{Y_k}(u(z)-u(y))^2dyds_z,
$$
Using Lemma \ref{integralest}, we have 
\begin{align*}
|X^*|\TwoNorm{u}{\wt{\omega}_{x}}^2&\leq \sum_{k\in {\cal I}}\sum_{j\in {\cal J}}\int_{F_{j}}\int_{Y_k}(u(z)-u(y))^2dyds_z\\
&\leq C\sum_{k\in {\cal I}}\sum_{j\in {\cal J}}   s_{k,j} \eta^{d+1}\TwoNorm{\nabla u}{P_{k,j}}^2\\
&\leq C s_{max} \eta^{d+1}\sum_{i\in {\cal I}}  \big|\{(k,j)\in {\cal I}\times {\cal J}: Y_{i}\subset P_{k,j}\}  \big|  \TwoNorm{\nabla u}{Y_{i}}^2\\
&\leq C s_{max} r_{max} \eta^{d+1} \TwoNorm{\nabla u}{\wt{\omega}_{x}}^2.
\end{align*}
Dividing by $|X^*|$ completes the argument.
\qed
\end{proof}

With the above lemmas we are able to obtain uniform bounds for isolated perforations. To illuminate the ideas needed to obtain such bounds, we propose a simple example. 
Again, as before, we suppose we have ${\omega}_{x}=[0,H]^2$.
 For simplicity of the exposition, we define the perforation solids as periodic square domains. More specifically we define the solid domain to be 
 $$S_\eta=\bigcup_{j=0}^{N_{\eta}}\bigcup_{i=0}^{N_{\eta}}\left([(2j+1)\eta,(2j+2)\eta] \times [(2i+1)\eta,(2i+2)\eta] \right),$$
here $N_{\eta}$ is chosen so that the top and right boundaries of ${\omega}_{x}$ are not intersected. We let $\{Y_{l}\}_{l=1}^n$ be a quasi-uniform and shape regular partition of $\wt{\omega}_{x}={\omega}_{x}\backslash S_{\eta}$ with mesh size $\eta>0$. We let $X^*=\{H\}\times [0,H]=\cup_{j=1}^{J}\bar{F}_j$ be the right boundary and $F_{j}$ the faces of  some elements in the partition. Thus, we  have $|X^*|=H$. Since the  particles are size $\eta$ and the minimal spacing is size $\eta$ there always exists a path from $Y_{k}\in \{Y_{l}\}_{l=1}^n$ to some face $F_j$ for some $j\in {\cal J}$ such that $s_{k,j}\leq C( \frac{H}{\eta}).$ In addition, it is also easy to see that $r_{max}\leq C(\frac{H}{\eta})^2$. Using the estimate in Lemma \ref{lemmaest2} we have 
\begin{align}\label{uniformbound}
C_{P}^2\leq C \frac{H}{\eta}  \frac{H^{2}}{\eta^2}   \frac{\eta^{3}}{H^3}\leq C,
\end{align}
or a uniform bound on the Poincar\'{e} constant for these isolated particles.
\begin{remark}
We note here that the periodicity of the particles is not at all essential on the above bound. Also the shape may be easily extended to convex and shape regular (i.e. not too oblique) isolated particles. The key parts being the ability to construct path lengths of order $s_{max}\leq C( \frac{H}{\eta}).$ We summarize this fact in a Corollary.
\end{remark}

\begin{corollary}\label{coruniform}
Suppose we have a collection  $S_{\eta}$ of isolated convex shape-regular particles with characteristic size and distance  $\eta>0$. Moreover,  suppose that $s_{max}\leq C( \frac{H}{\eta})$ and  $r_{max}\leq C(\frac{H}{\eta})^2$ for some quasi-uniform and shape regular partition $\{Y_{l}\}_{l=1}^{n}$,  with mesh size $\eta>0$ of $\wt{\omega}_{x}={\omega}_{x}\backslash S_{\eta}$. Then, the Poincar\'{e} constant is uniformly bounded
\begin{align}\label{coruniformbound}
C_{P}\leq C,
\end{align}
where $C$ does not depend on $H$ or $\eta$.
\end{corollary}

Finally, applying this estimate to Theorem \ref{errorlocal} we have 
\begin{align}
\TwoNorm{\nabla u-\nabla u_{H,m}^{ms}}{\perf}\leq C\left(H+ m^{\frac{d}{2}}  \theta^{m  } \right) \TwoNorm{g}{\perf},
\end{align}
and recover the standard estimate as in \cite{MP11} independent of the small structures $\eta$. Note here that the above $\theta$ is also independent of $H$ and $\eta.$

\section{Numerical Examples}
In this section we will present a two numerical examples. We apply our algorithm to \eqref{perfLaplace} using our multiscale method and compare with standard $\mathbb{P}_1$ finite elements. Using a penalization method, we will  implement the micro-structural features into the domain.
We will do this for two relevant examples. The first being a periodic square domain with square particles, and the second an dumbbell-shaped domain containing the microstructure of the first experiment. We will demonstrate the validity of  our estimates based on varying patch size (truncation of the localization) and by varying microstructure lengths $\eta$. When we vary the microstructure lengths we will also fix our truncation patch parameter $k$ to be proportional to $\text{log}(H).$

We begin by describing the geometry of the domains. First, we take our unperforated domain to be $\Omega=[0,1]^2$ and define the unit cell to be $Y=[0,1]^2\backslash[\frac{1}{4},\frac{3}{4}]^2$. We define the perforated domain to be
\begin{align}\label{periodicdomain}
\perf^\eta=\bigcup_{k\in \mathbb{Z}^2}\left(\eta (Y+k)\right)\cap \Omega,
\end{align}
where $\eta$ is chosen so that the domain is periodically tessellated. This domain for $\eta=\frac{1}{8}$ can be seen in Figure \ref{geom1}.

  \begin{figure}[!ht]
  \centering
      \includegraphics[width=0.5\textwidth]{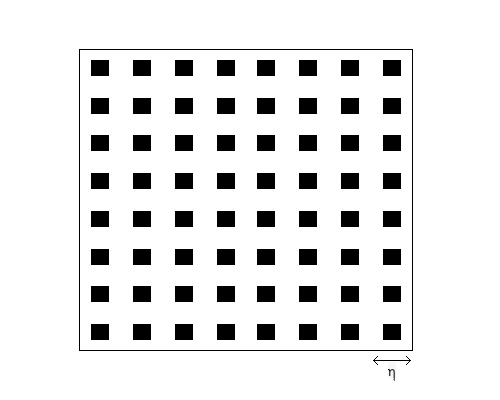}
       \caption{Square domain with periodic microstructure.}
    \label{geom1}
  \end{figure}

Since this geometry will clearly be in the same class as the uniform bound estimate \eqref{coruniformbound}, we choose our second geometry to be an dumbbell-shaped domain. As noted in \cite{Pechstein01042013}, such a  shaped domain has a theoretical bound $C_{P}^2\leq 1+\text{log}(\frac{\text{diam}(\Omega)}{\eta})$. Here $\eta$ is the separation of the narrowest part of the domain. More concretely, we let 
$$\Omega^{H,\eta}=\Omega\backslash\left( \left([\frac{3}{8},\frac{5}{8}]\times [0,\frac{1-\eta}{2}]\right)\cup\left([\frac{3}{8},\frac{5}{8}]\times [\frac{1+\eta}{2},1]\right) \right).$$
 In addition to the $H$ structure we also take out some of the square perforations as in \eqref{periodicdomain} for a fixed period of $\frac{1}{16}$. We define the following domain
\begin{align}\label{Hdomain}
\perf^{H,\eta}=\bigcup_{k\in \mathbb{Z}^2}\left(\frac{1}{16} (Y+k)\right)\cap \Omega^{H,\eta},
\end{align}
this domain can be seen in Figure \ref{geom2}. Note here, the size of the perforations are fixed and the varying quantity is the size of the narrowest part of the domain in the middle strip.

  \begin{figure}[!ht]
  \centering
      \includegraphics[width=0.45\textwidth]{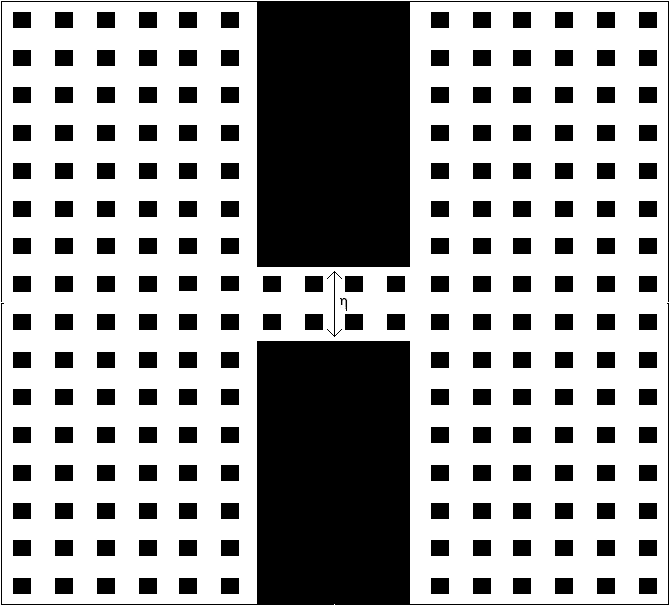}
       \caption{Dumbbell-shaped domain with periodic microstructure}
    \label{geom2}
  \end{figure}

To solve the problems in the porous domains, we will  explicitly grid the perforations on the fine scale, not on the coarse scale. 
A penalization scheme could also be utilized to relax the restrictiveness of gridding the fine scale. 
Note, there is a fine scale $h$ to solve the local problems and we take this value to be $h=2^{-8}$.
For all the following examples we will use the forcing 

\begin{displaymath}
   g(x_{1}, x_{2}) = \left\{
     \begin{array}{lr}
       1 , x_{2}\geq .5\\
       0 , x_{2}< .5.
     \end{array}
   \right.
\end{displaymath} 

In addition to using the projective quasi-interpolation operator \eqref{proj}, we also present results from the Cl\'{e}ment interpolation operator  \cite{MP11},
\begin{align}\label{clementproj}
\Qintp u=\sum_{x\in{\cal N}_{N}} (\wt{{\cal I}}_{H} u)(x){\cal R}  \lambda_{x},
\end{align}
where $(\wt{{\cal I}}_{H} u)(x)=\frac{ \int_{\perf} v \lambda_x dz  }{\int_{\perf} \lambda_x dz}$. Recall, we chose the projective quasi-interpolator only to simplify the proofs, and here we present numerical results to show that, in these cases, good results hold for the Cl\'{e}ment interpolation operator also.

%

  \begin{figure}[!ht]
    \subfloat[Varying $\eta$ for patch size $k\approx\text{log}(H)$. \label{subfig-11}]{%
      \includegraphics[width=0.45\textwidth]{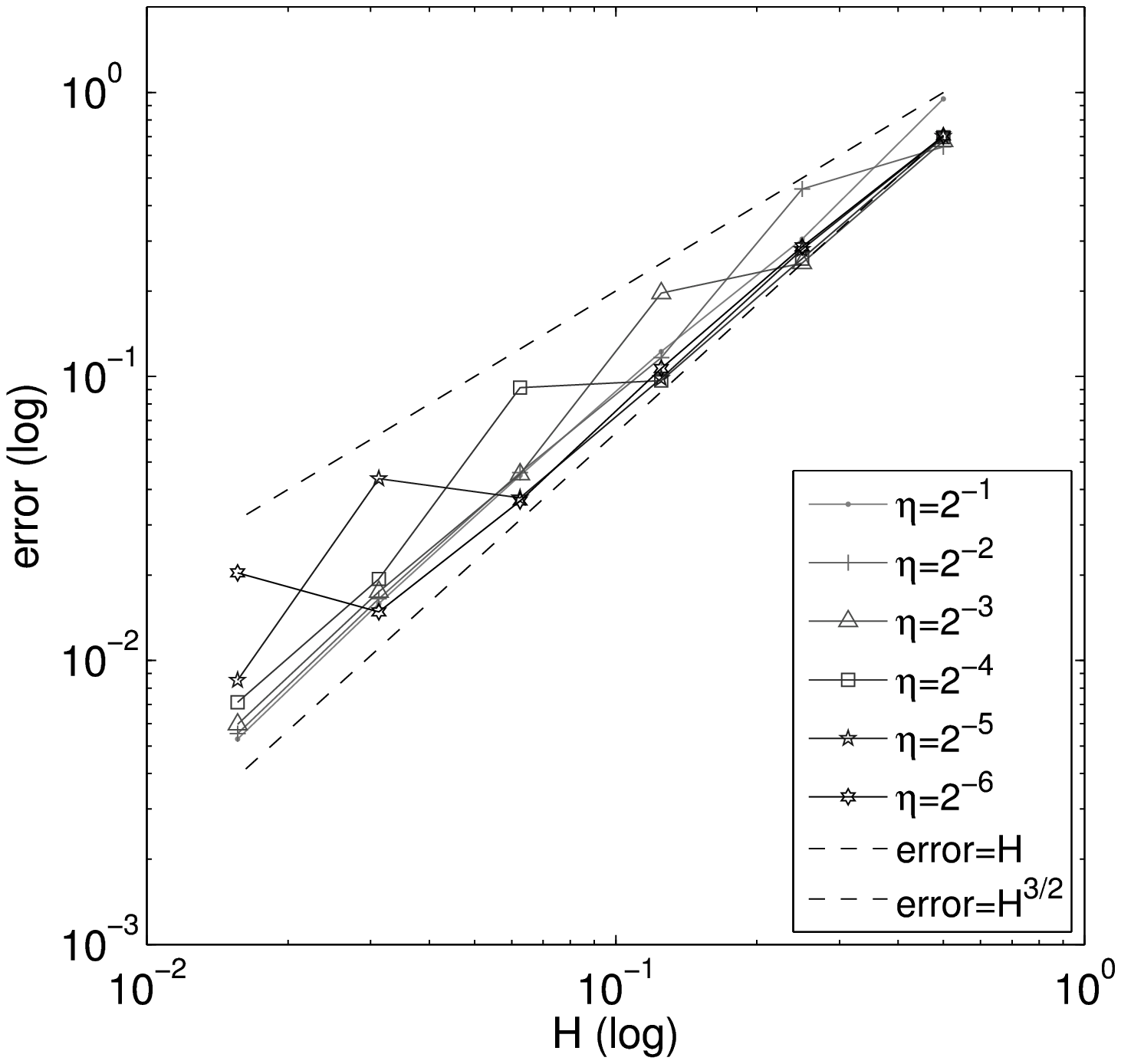}
    }
    \hfill
    \subfloat[Varying patch size $k$ for $\eta=2^{-6}$.\label{subfig-12}]{%
      \includegraphics[width=0.45\textwidth]{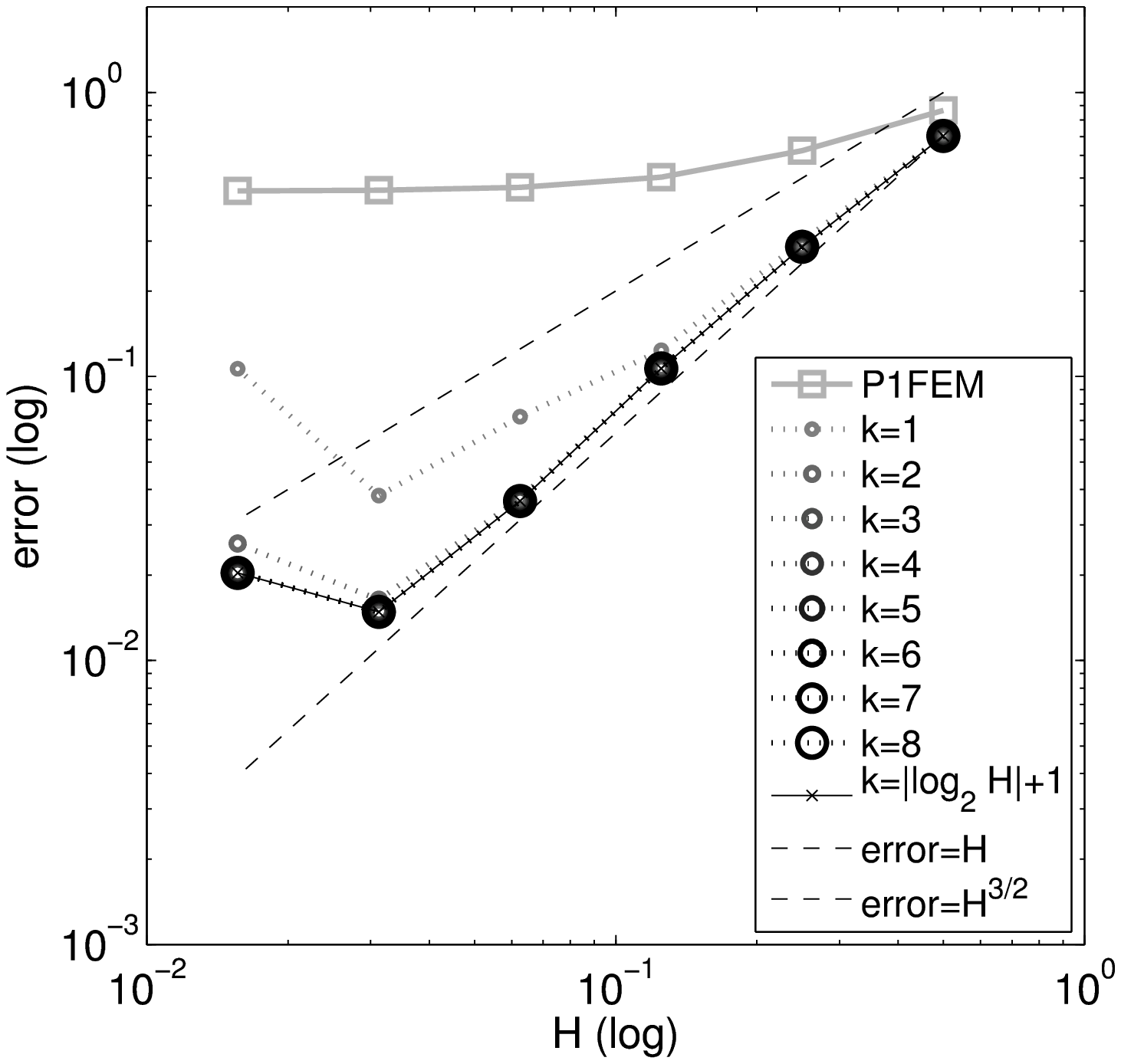}
    }
    \caption{Results for example geometry in Figure \ref{geom1}, using projective quasi-interpolation \eqref{proj}. }
    \label{fig:11}
  \end{figure}

  \begin{figure}[!ht]
    \subfloat[Varying $\eta$ for patch size $k\approx\text{log}(H)$. \label{subfig-13}]{%
      \includegraphics[width=0.45\textwidth]{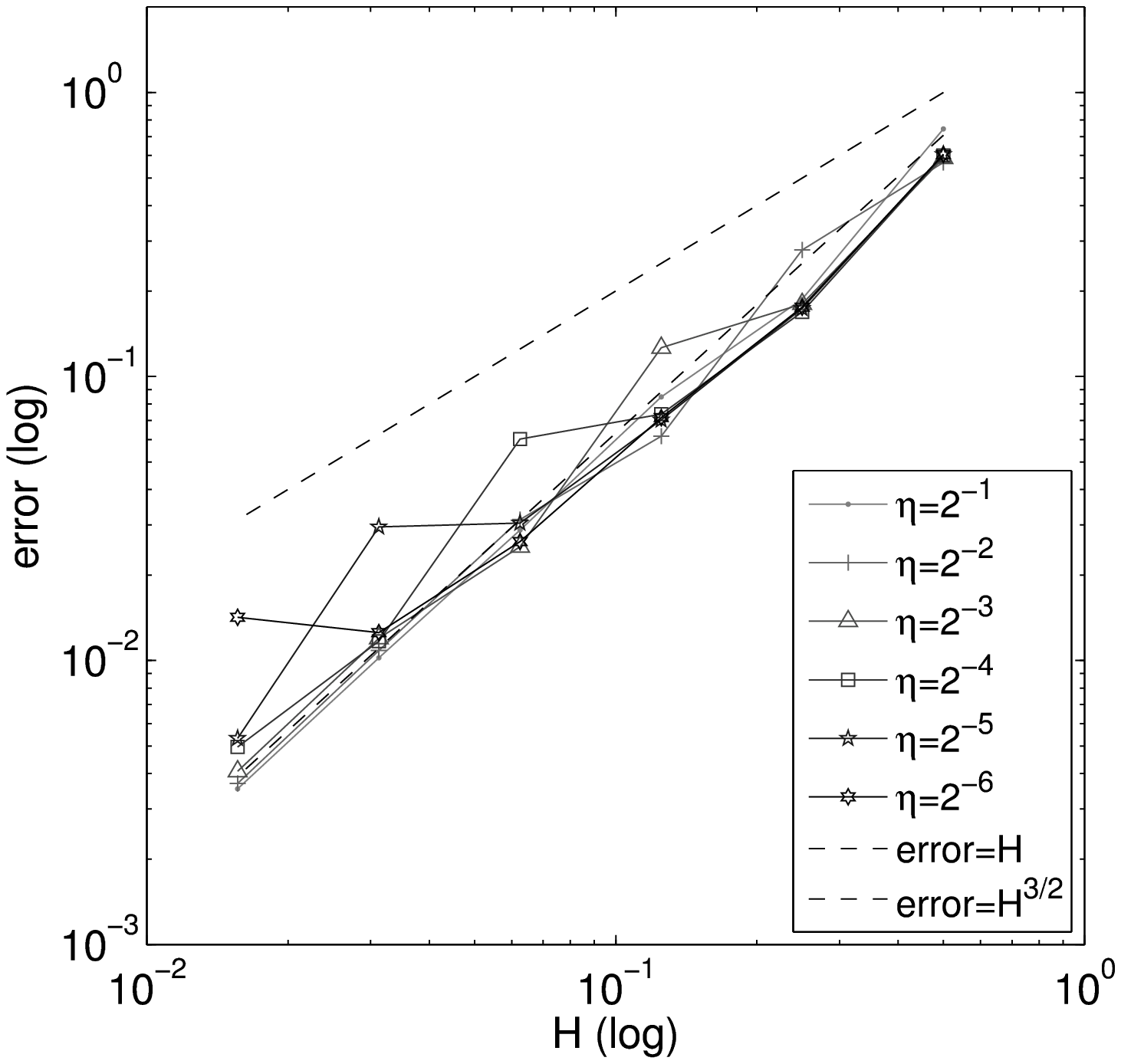}
    }
    \hfill
    \subfloat[Varying patch size $k$ for $\eta=2^{-6}$.\label{subfig-14}]{%
      \includegraphics[width=0.45\textwidth]{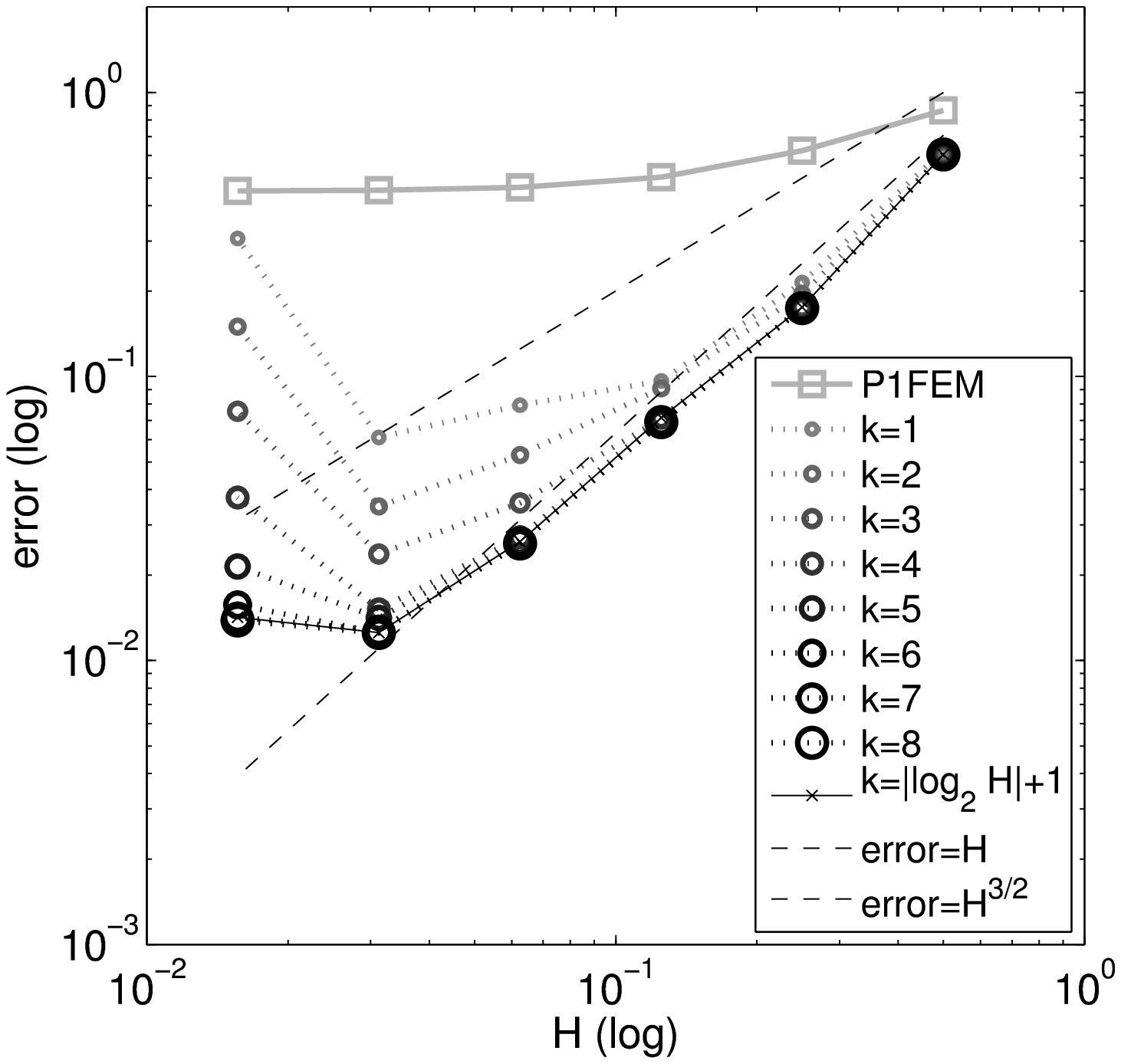}
    }
    \caption{Results for example geometry in Figure \ref{geom1}, using Cl\'{e}ment quasi-interpolation \eqref{clementproj}. }
    \label{fig:12}
  \end{figure}


  \begin{figure}[!ht]
    \subfloat[Varying $\eta$ for patch size $k\approx\text{log}(H)$. \label{subfig-21}]{%
      \includegraphics[width=0.45\textwidth]{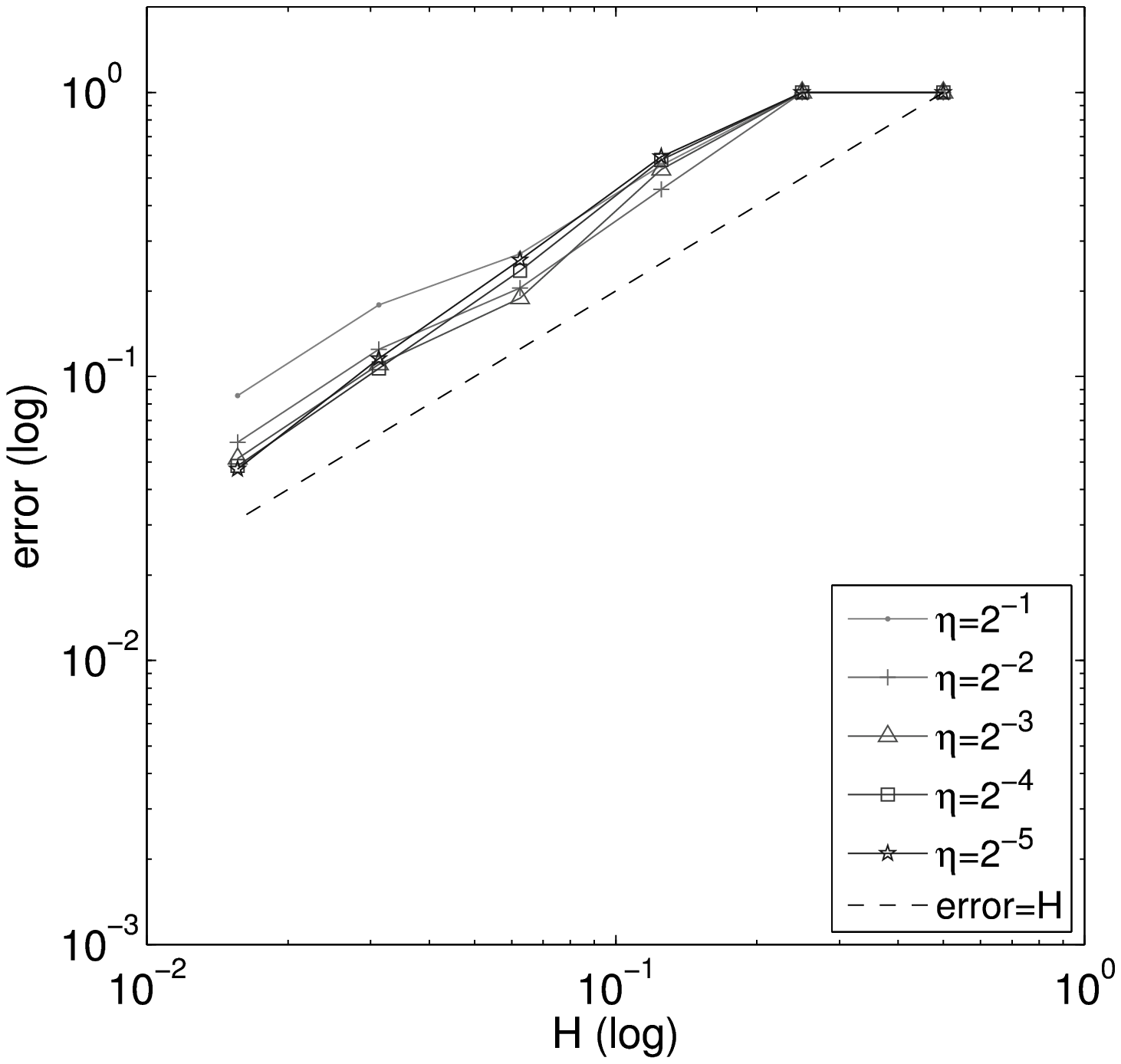}
    }
    \hfill
    \subfloat[Varying patch size $k$ for $\eta=2^{-6}$.\label{subfig-22}]{%
      \includegraphics[width=0.45\textwidth]{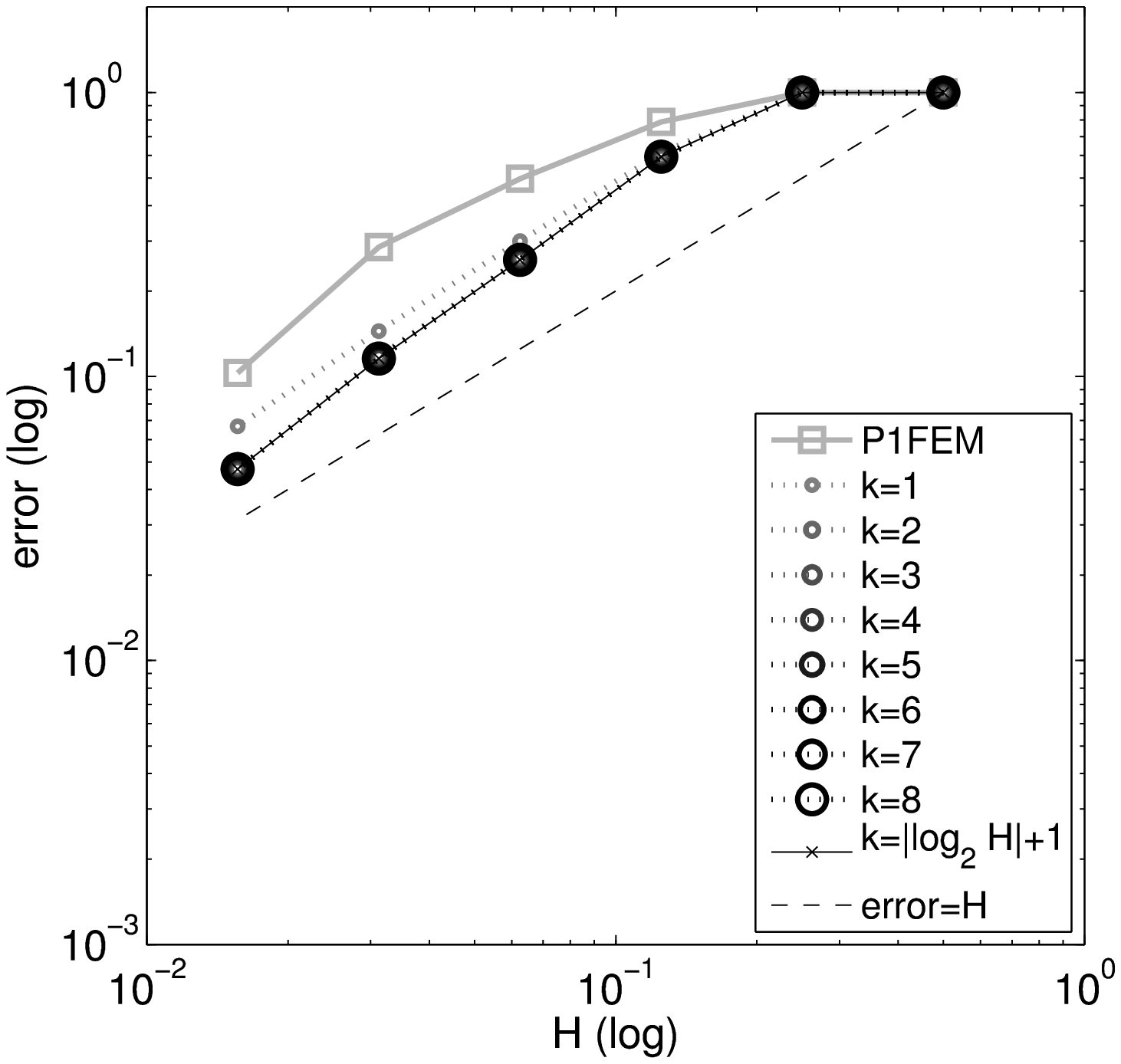}
    }
    \caption{Results for example geometry in Figure \ref{geom2}, using projective quasi-interpolation \eqref{proj}. }
    \label{fig:21}
  \end{figure}

  \begin{figure}[!ht]
    \subfloat[Varying $\eta$ for patch size $k\approx\text{log}(H)$. \label{subfig-23}]{%
      \includegraphics[width=0.45\textwidth]{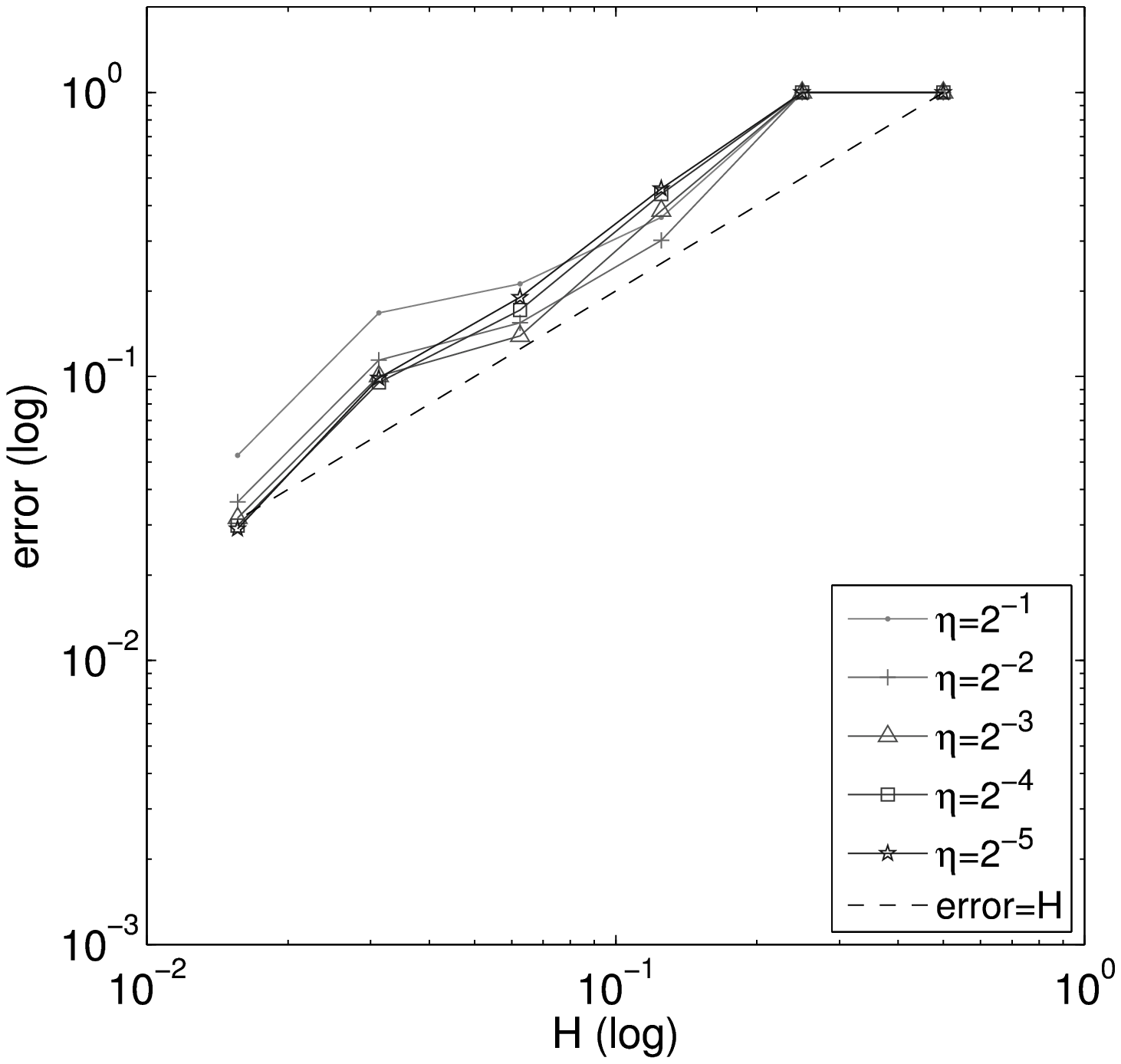}
    }
    \hfill
    \subfloat[Varying patch size $k$ for $\eta=2^{-6}$.\label{subfig-24}]{%
      \includegraphics[width=0.45\textwidth]{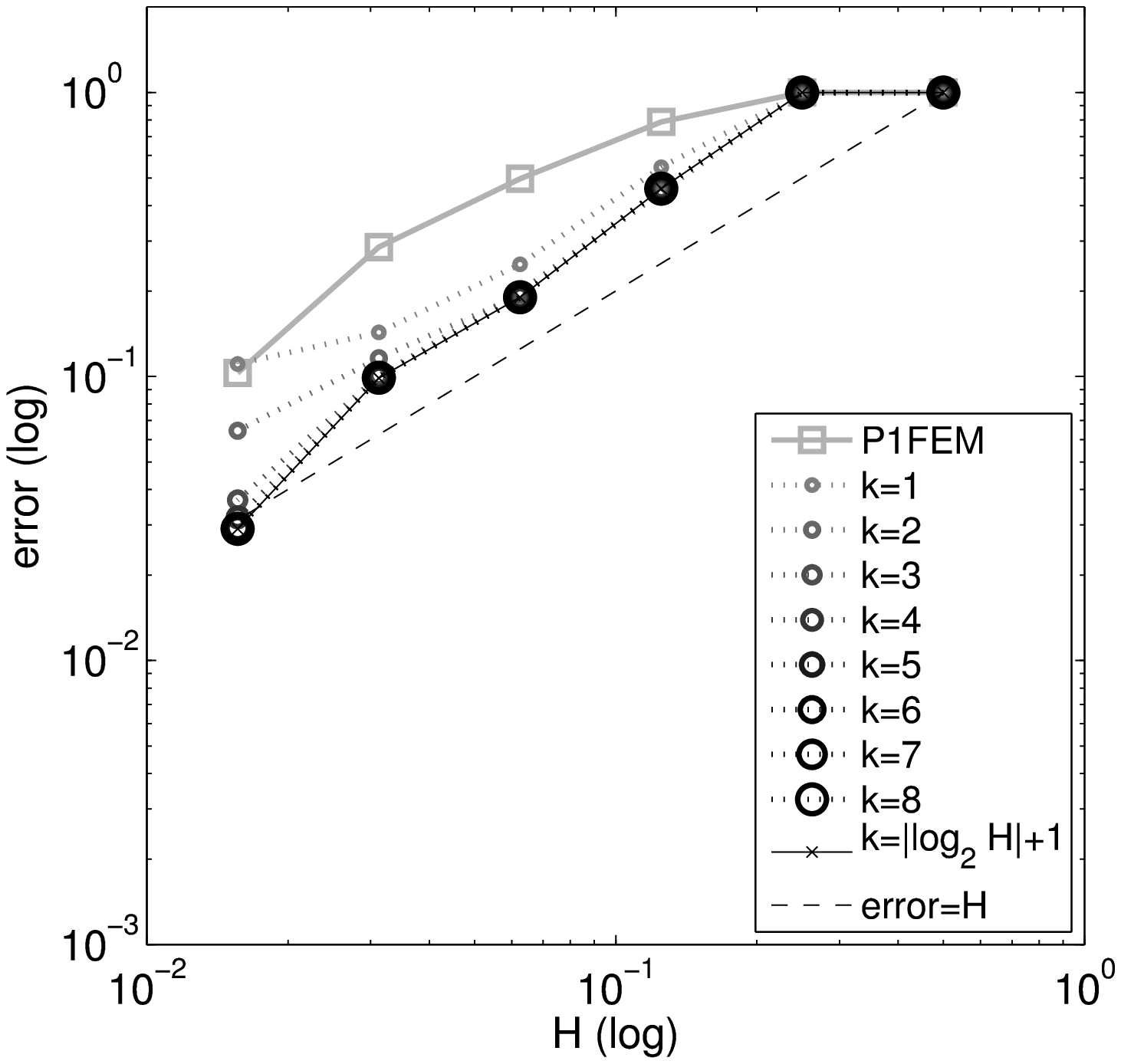}
    }
    \caption{Results for example geometry in Figure \ref{geom2}, using Cl\'{e}ment quasi-interpolation \eqref{clementproj}. }
    \label{fig:22}
  \end{figure}

We present results for both media \eqref{periodicdomain} and \eqref{Hdomain} while using both interpolation operators \eqref{proj} and \eqref{clementproj}. We have two types of numerical tests. First, varying the microstructure parameter $\eta$ while keeping the $k$-patch growth fixed to $\text{log}(H)$. The idea here to see the effect of the error estimates from the possibly error degrading Poincar\'{e} constant. Second, we fix the microstructure length to the smallest value and vary the patch size to observe the rates of exponential convergence. All of these results are compared against an "overkill"   fine-scale solution with $h=2^{-8}$ in the $H^1$ norm.

The results from the geometry \eqref{periodicdomain} are contained in Figure \ref{fig:11} and \ref{fig:12}. In Figure \ref{fig:11}, we use the projective interpolator  \eqref{proj}  and in Figure \ref{fig:12} we use the Cl\'{e}ment interpolator \eqref{clementproj}. Varying the microstructure, in the case period size $\eta$, while fixing the patch extension $k\approx \text{log}(H)$ we plot the results for both interpolators in Figure \ref{subfig-11} and Figure \ref{subfig-13}. In both examples we see that the Poincar\'{e} constant does not effect the estimate negatively in agreement with \eqref{coruniformbound}. In Figure \ref{subfig-12} and Figure \ref{subfig-14} we fix the geometric parameter to the smallest value $\eta=2^{-6}$ and vary the patch size parameter $k$. We note that the slightly more expensive projective quasi-interpolator, as it requires many local $L^2$ projections, performs better in this case at exponential convergence of the patch extensions. 

The results from the geometry \eqref{Hdomain} are contained in Figure \ref{fig:21} and \ref{fig:22}. In Figure \ref{fig:21}, we use the projective interpolator  \eqref{proj}  and in Figure \ref{fig:22} we use the Cl\'{e}ment interpolator \eqref{clementproj}. Keeping the period fixed but varying  $\eta$, the width of the thinnest part, 
and again fixing the patch extension $k\approx \text{log}(H)$ we plot the results for both interpolators in Figure \ref{subfig-21} and Figure \ref{subfig-23}. 
%
%
%
 In Figure \ref{subfig-22} and Figure \ref{subfig-24} we fix the geometric parameter to the smallest value $\eta=2^{-6}$ and vary the patch size parameter $k$. Again we see slightly  better performance  with respect to exponential convergence of the patch extensions for the projective quasi-interpolation. 

\section{Conclusion}

In this work we developed a multiscale procedure to compute Laplacian problems with zero Neumann data in domains with complicated porous microstructure. We are were able to determine the error with respect to the ideal corrector and error due to truncation and localization of the multiscale correctors. As was noted, keeping track of Poincar\'{e} constants was critical in our analysis as they may contain information about the microstructure. We used a constructive procedure to estimate these constants and obtain bounds with respect to $H$ and $\eta$. This procedure was demonstrated on two interesting examples. Finally, we implemented numerical tests to validate our theoretical estimates. We found our numerical experiments were in agreement with the theory and the quasi-interpolator based on local $L^2$ projections to perform slightly better than the Cl\'{e}ment type.

\appendix

\section{Quasi-Interpolation Stability}\label{quasiproof}


We now will prove the stability estimate used throughout for this projective quasi-interpolation operator \eqref{proj}. The proof of this lemma is based on that presented in \cite{OhlbergerNotes}.
\begin{lemma}
  For $u\in \Hperf$, there exists a constant $\local>0$, such that 
   \begin{align}\label{stableproj}
 &H^{-1}\norm{u-\Qintp u}_{L^2(\wt{\omega}_{x,0} ) }+\norm{\nabla (u-\Qintp u)}_{L^2(\wt{\omega}_{x,0})} \leq  \local  \norm{\nabla  u}_{L^2(\wt{\omega}_{x,1} )},
\end{align}
where $\local=C C_{P}$. Here $C_{P}$ is the Poincar\'{e} constant and $C$ is a benign  constant not depending on $H$ or $\eta$. Morever, the interpolation $\Qintp$ is a projection. 
\end{lemma}
\begin{proof}
Note that we have easily from this definition taking $v_{H}=({\cal P}_{x} u)$
 and applying Cauchy-Schwarz 
thus, $\TwoNorm{{\cal P}_{x}u}{\wt{\omega}_{x}}\leq \TwoNorm{u}{\wt{\omega}_{x}}$. We use $u-\avrg{u}{\wt{\omega}_{x}}$, here again  $\avrg{u}{\wt{\omega}_{x}}=\frac{1}{|\wt{\omega}_x|}\int_{\wt{\omega}_x}udz,$ the fact that the $L^2$ projection of a constant is itself, and the fact that $(1-{\cal P}_{x})$ is also a projection we obtain 
\begin{align}\label{localstab}
 \norm{u-{\cal P}_x u }_{L^2(\wt{\omega}_{x})}\leq& \norm{u-\avrg{u}{\wt{\omega}_{x}} }_{L^2(\wt{\omega}_{x})}\leq H C_P  \norm{\nabla u }_{L^2(\wt{\omega}_{x})}.
\end{align}
Here, we used the inequality \eqref{poincarelocal} to obtain the gradient bound. To obtain the derivative bound note that by a use of the inverse inequality and \eqref{poincarelocal} we have 
\begin{align}
 \norm{\nabla (u-{\cal P}_x u) }_{L^2(\wt{\omega}_{x})}
\label{localstabder}
& \leq (1+C C_{p}) \norm{\nabla u }_{L^2(\wt{\omega}_{x})}.
\end{align}
This is merely the $H^1$ stability of the $L^2$ projection c.f. \cite{projstable} and references therein.

We suppose that the basis functions form a partition of unity, that is $\sum_{x\in{\cal N}_{H}}\lambda_x=1$. We are only proving for the elements that do not meet the boundary. If the elements meet the boundary the Friedrichs' inequality can be utilized. 
%
%
Thus, we have for the $L^2$ norm
\begin{align}
\norm{u-\Qintp u}_{L^2(\wt{\omega}_{x} )}&=\norm{u-\sum_{x\in{\cal N}_{H}} ({\cal P}_{x} u)(x) \lambda_{x}}_{L^2(\wt{\omega}_{x} )}\nonumber\\
&=\norm{\sum_{x\in{\cal N}_{H}}(u- ({\cal P}_{x} u)(x)) \lambda_{x}}_{L^2(\wt{\omega}_{x}  )}\nonumber\\
&\leq \sum_{x\in{\cal N}_{H}} \norm{u- ({\cal P}_{x} u)(x)  }_{L^2(\wt{\omega}_{x} )}\nonumber\\
\label{inequ1}
&\leq \sum_{x\in{\cal N}_{H}} \norm{u- {\cal P}_{x} u  }_{L^2(\wt{\omega}_{x} )}+\sum_{x\in{\cal N}_{H}} \norm{{\cal P}_{x} u-({\cal P}_{x} u)(x)  }_{L^2(\wt{\omega}_{x} )}.
\end{align}
We can easily estimate the first term by using \eqref{localstab}, taking a closer look at the second term, again using the partition of unity property, we have 
\begin{align*}
\norm{{\cal P}_{x} u-({\cal P}_{x} u)(x)  }_{L^2(\wt{\omega}_{x}  )}&=\norm{\sum_{x'}\left( ({\cal P}_{x} u)(x')-({\cal P}_{x} u)(x)  \right)\lambda_{x'}  }_{L^2(\wt{\omega}_{x} )}\\
&\leq  \sum_{x'}\norm{ ({\cal P}_{x} u)(x')-({\cal P}_{x} u)(x)  }_{L^2(\wt{\omega}_{x} )}\\
&\leq  \sum_{x'} |\wt{\omega}_{x,1} |^{\frac{1}{2}}|   ({\cal P}_{x} u)(x')-({\cal P}_{x} u)(x)  |\\
& \leq C \sum_{x'} |\wt{\omega}_{x,1} |^{\frac{1}{2}}  H_{}  \norm{ \nabla ({\cal P}_{x} u)(x)  }_{L^{\infty}(\wt{\omega}_{x} )}\\
& \leq  C \sum_{x'} H_{}  \norm{ \nabla ({\cal P}_{x} u)(x)  }_{L^{2}(\wt{\omega}_{x} )}
\end{align*}
Returning to \eqref{inequ1}, we have 
\begin{align}
\norm{u-\Qintp u}_{L^2(\wt{\omega}_{x} )} & \leq \sum_{x\in{\cal N}_{H}} \norm{u- {\cal P}_{x} u  }_{L^2(\wt{\omega}_{x}) }+\sum_{x\in{\cal N}_{H}} \norm{{\cal P}_{x} u-({\cal P}_{x} u)(x)  }_{L^2(\wt{\omega}_{x} )}\nonumber\\
&\leq \sum_{x\in {\cal N}_{H} }H C_{P}  \norm{\nabla u }_{L^2(\wt{\omega}_{x,1})}+C \sum_{x'\in {\cal N}_{H}}  H_{}  \norm{ \nabla ({\cal P}_{x} u)(x)  }_{L^{2}(\wt{\omega}_{x} )}\nonumber\\
\label{L2Est}
&\leq C C_{p} H_{} \norm{\nabla u }_{L^2(\wt{\omega}_{x,1})}.
\end{align}
Using the estimate \eqref{localstabder}, and a similar argument as above for the $L^2$ estimate  \cite{OhlbergerNotes}, we obtain the derivative estimate 
\begin{align}
\label{L2Est}
\norm{\nabla (u-\Qintp u)}_{L^2(\wt{\omega}_{x} )}  \leq C C_{p}  \norm{\nabla u }_{L^2(\wt{\omega}_{x,1})}.
\end{align}
To see the $\Qintp$ is a projection note  for ${\cal P}_{x}$, the local patch $L^2$ projection, acting on ${\cal R} \lambda_x$ is  a projection, and moreover is identity. By definition we have
\begin{align}
\int_{\wt{\omega}_{x}}( {\cal P}^2_{x}\lambda_{x}) v_{H}dz= \int_{\wt{\omega}_{x}}\lambda_x v_{H}dz \text{ for all } v_{H}\in  \wt{V}_{H}(\wt{\omega}_{x}),
\end{align}
and thus it is trivial to see ${\cal P}^2_{x}\lambda_{x}={\cal P}_{x}\lambda_{x}=\lambda_{x}$ on $\wt{\omega}_{x}$ for all $x\in {\cal N}_{H}$. Thus, 
$$\Qintp({\cal R} \lambda_{x})=\sum_{x'\in{\cal N}_{H}} ({\cal P}_{x'}  ({\cal R}\lambda_{x'} )  )(x){\cal R}  \lambda_{x}=\sum_{x'\in{\cal N}_{H}} ( {\cal R}\lambda_{x'}   )(x'){\cal R}  \lambda_{x}={\cal R}  \lambda_{x},$$
and so $\Qintp^2({\cal R}\lambda_{x})=\Qintp({\cal R}\lambda_{x})={\cal R}\lambda_{x},$ and so by linearity
$$\Qintp^2(u)=\Qintp\left(\sum_{x\in{\cal N}_{H}} ({\cal P}_{x}  u)(x) {\cal R}  \lambda_{x}\right)=\sum_{x\in{\cal N}_{H}} ({\cal P}_{x}  u)(x)\Qintp( {\cal R}  \lambda_{x})=\sum_{x\in{\cal N}_{H}} ({\cal P}_{x}  u)(x) {\cal R}  \lambda_{x}.$$ From here we see that $\Qintp^2=\Qintp$.
\qed
\end{proof}

\section{Auxiliary Lemmas}\label{auxsection}

Now we will prove and state the auxiliary lemmas used to prove estimate \eqref{eq:errorglobal}.  These proofs are largely based on the works \cite{Henning.Morgenstern.Peterseim:2014,MP11} and references therein. However, here we must carefully track the occurrence of Poincar\'{e} constants.

First, we begin with the quasi-incusion property. For $x,x' \in{\cal N}_{H}$ and $l,k\in \mathbb{N}$ and $m=0,1,\cdots,$ with $k\geq l\geq 2$ we have if
\begin{align}\label{quasiinclusion}
\wt{\omega}_{x',m+1}\cap \left(\wt{\omega}_{x,k}\backslash \wt{\omega}_{x,l} \right)\neq \emptyset, \text{ then }\wt{\omega}_{x',1}\subset\left(\wt{\omega}_{x,k+m+1}\backslash \wt{\omega}_{x,l-m-1}. \right)
\end{align}

We will use the cutoff functions defined in \cite{Henning.Morgenstern.Peterseim:2014}. 
 For $x\in {\cal N}_{H}$ and $k>l\in \mathbb{N}$, let $\eta^{k,l}_{x}:\perf \to [0,1]$ be a continuous weakly differentiable functions so that 
 \begin{subequations}\label{cutoff1}
\begin{align}
\left(   \eta^{k,l}_{x}  \right)|_{\wt{\omega}_{x,k-l}}&=0,\\
\left(   \eta^{k,l}_{x}  \right)|_{\perf \backslash \wt{\omega}_{x,k}}&=1,\\
\forall T\in {\cal T}_{H}, \norm{\nabla \eta^{k,l}_x}_{L^{\infty}(T)}&\leq C_{}\frac{1}{l H_{}}.
\end{align}
\end{subequations}
A precise form of $\eta_{x}^{k,l}$ can be written as 
$$
\eta_{x}^{k,l}(x')=\frac{ \text{dist}(x',\wt{\omega}_{x,k-l})  }{   \text{dist}(x',\wt{\omega}_{x',k-l}) +\text{dist}(x',\perf \backslash \wt{\omega}_{x,k})  }. 
$$
If $\perf \backslash \wt{\omega}_{x,k}=\emptyset$, then we prescribe $\eta^{k,j}_{x}=0$.

 Unlike in \cite{Henning.Morgenstern.Peterseim:2014}, we are using a quasi-interpolation that is also a projection. This simplifies the proofs since there is no need to construct an approximate projection. 
Here we will need the following simplified quasi-invariance of the fine-scale space under multiplication by cutoff functions. We write this estimate in the following lemma. 
 
 \begin{lemma}\label{qi}
 Let $k>l\in \mathbb{N}$ and  $x\in {\cal N}_{H}$. Suppose that $w\in \wt{V}^f$, then we have the estimate
 \begin{align}
 \TwoNorm{\nabla \Qintp(\eta_{x}^{k,l} w)   }{\perf}\leq C_{1} l^{-1}\TwoNorm{ \nabla w }{ \wt{\omega}_{x,k+2}\backslash \wt{\omega}_{x,k-l-2} },
  \end{align}
  here $C^2_{1}=C^2_{lip} \local +\local^3$.
 \end{lemma}
 
 \begin{proof}
Fixing $x$ and $k$, we denote the average  as $\avrg{\eta_{x}^{k,l}}{\wt{\omega}_{x',1} }=\frac{1}{|\wt{\omega}_{x',1}|} \int_{\wt{\omega}_{x',1}} \eta_{x}^{k,l}dz. $ We estimate on a single patch $\wt{\omega}_{x}$, using the fact that $\Qintp(w)=0$ and the estimate \eqref{stableproj} we have 

\begin{align*}
&\TwoNorm{\nabla \Qintp(\eta_{x}^{k,l}w )}{\wt{\omega}_{x'}}=\TwoNorm{\nabla \Qintp((\eta_{x}^{k,l}-\avrg{\eta_{x}^{k,l}}{\wt{\omega}_{x',1}}) w)}{\wt{\omega}_{x'}}\\
&\leq \local\TwoNorm{\nabla ( (\eta_{x}^{k,l}-\avrg{\eta_{x}^{k,l}}{\wt{\omega}_{x',1}})  w)  }{\wt{\omega}_{x',1}}\\
&\leq  \local\left(    \TwoNorm{(\eta_{x}^{k,l}-\avrg{\eta_{x}^{k,l}}{\wt{\omega}_{x',1}})\nabla w }{ \wt{\omega}_{x',1} }    +   \TwoNorm{\nabla  \eta^{k,l}_{x} (w-\Qintp(w)) }{ \wt{\omega}_{x',1}}  \right).
%
\end{align*}
Summing over all $x\in {\cal N}_{H}$, using the quasi-inclusion property \eqref{quasiinclusion}, and the above calculation yields 
\begin{align*}
 \TwoNorm{\nabla \Qintp(\eta_{x}^{k,l} w)   }{\perf}^2 & \leq \sum_{x'\in {\cal N}_{H}}  \TwoNorm{\nabla \Qintp(\eta_{x}^{k,l} w)   }{\wt{\omega}_{x'} }^2\\
 &\leq \local \sum_{\wt{\omega}_{x'}\subset\wt{\omega}_{x,k+2}\backslash \wt{\omega}_{x,k-l-2} }   \TwoNorm{\nabla ( (\eta_{x}^{k,l}-\avrg{\eta_{x}^{k,l}}{\wt{\omega}_{x',1}})  w)  }{\wt{\omega}_{x',1}}^2\\
 &\leq \local \sum_{\wt{\omega}_{x'}\subset\wt{\omega}_{x,k+2}\backslash \wt{\omega}_{x,k-l-2} }      \TwoNorm{ (\eta_{x}^{k,l}-\avrg{\eta_{x}^{k,l}}{\wt{\omega}_{x',1}})\nabla w }{ \wt{\omega}_{x',1} }^2      \\
 &+ \local \sum_{\wt{\omega}_{x'}\subset\wt{\omega}_{x,k+2}\backslash \wt{\omega}_{x,k-l-2} }     \TwoNorm{\nabla  \eta^{k,l}_{x}( w-\Qintp(w)) }{ \wt{\omega}_{x',1}}^2.
 \end{align*}
 Noting that $\nabla \eta_{x}^{k,l}\neq 0$ only in $\wt{\omega}_{x,k}\backslash\wt{\omega}_{x,k-l} $ and $(\eta_{x}^{k,l}-\avrg{\eta_{x}^{k,l}}{\wt{\omega}_{x',1}})\neq0$ only if $\wt{\omega}_{x',k}$ intersects $\wt{\omega}_{x,k}\backslash\wt{\omega}_{x,k-l} $ hence we obtain the tighter estimate 
 
\begin{align*}
 \TwoNorm{\nabla \Qintp(\eta_{x}^{k,l} w)   }{\perf}^2 &\leq \local \sum_{\wt{\omega}_{x'}\subset\wt{\omega}_{x,k+1}\backslash \wt{\omega}_{x,k-l-1} }      \TwoNorm{ (\eta_{x}^{k,l}-\avrg{\eta_{x}^{k,l}}{\wt{\omega}_{x',1}})\nabla w }{ \wt{\omega}_{x',1} }^2    \\
 &+ \local \sum_{\wt{\omega}_{x'}\subset\wt{\omega}_{x,k+1}\backslash \wt{\omega}_{x,k-l-1} }     \TwoNorm{\nabla  \eta^{k,l}_{x}( w-\Qintp(w)) }{ \wt{\omega}_{x',1}}^2  .
 \end{align*}
Using the Lipschitz bound $\norm{\eta_{x}^{k,l}-\avrg{\eta_{x}^{k,l}}{\wt{\omega}_{x',1}}}_{L^{\infty}(\wt{\omega}_{x',1})} \leq C_{lip}H \norm{\nabla \eta_{x}^{k,l} }_{L^{\infty}(\wt{\omega}_{x',1})} $ on the first term and  \eqref{stableproj} on the second we obtain

\begin{align*}
 \TwoNorm{\nabla \Qintp(\eta_{x}^{k,l} w)   }{\perf}^2 &\leq C^2_{lip} \local
   H^2 \norm{\nabla \eta_{x}^{k,l} }_{L^{\infty}(\perf)}^2  \TwoNorm{ \nabla w }{ \wt{\omega}_{x,k+1}\backslash \wt{\omega}_{x,k-l-1} }^2    \\
 &+ \local^3  H^2 \norm{\nabla \eta_{x}^{k,l} }_{L^{\infty}(\perf)}^2  \TwoNorm{ \nabla w }{ \wt{\omega}_{x,k+1}\backslash \wt{\omega}_{x,k-l-1} }^2 .
 \end{align*}
 Finally, taking another layer on the outside and inside of the annulus patch we arrive at 
 \begin{align*}
 \TwoNorm{\nabla \Qintp(\eta_{x}^{k,l} w)   }{\perf}^2 \leq l^{-2}(C^2_{lip} \local +\local^3 )
    \TwoNorm{ \nabla w }{ \wt{\omega}_{x,k+2}\backslash \wt{\omega}_{x,k-l-2} }^2,
    \end{align*}
    here $C^2_{1}=C^2_{lip} \local +\local^3$, and note that  $C_{lip} \leq C C_{P}$.
 \qed
 \end{proof}
 
 We now will demonstrate the decay of the fine-scale space in the next lemma. 
 
\begin{lemma}\label{decaylemma}
Fix some $x\in {\cal N}_{H}$ and $F\in (\wt{V}^f)'$ the dual of $\wt{V}^f$ satisfying $F(w)=0$ for all $w\in \wt{V}^f(\perf\backslash \wt{\omega}_{x,1})$.  Then, for $u\in \wt{V}^F$ the solution of 
\begin{align}
\int_{\perf} \nabla u \nabla w dz =F(w) \text{ for all } w\in \wt{V}^f.
\end{align}
Then, there exists a constant $\theta\in (0,1)$ such that for $k\in \mathbb{N}$  we have 
\begin{align}
\TwoNorm{\nabla u}{\perf \backslash \wt{\omega}_{x,k}}\leq \theta^k \TwoNorm{\nabla u}{\perf }.
\end{align}
We have $\theta=e^{-\frac{1}{\lceil C_{2} e \rceil+2}}\in (0,1)$,  here  $C_{2}=(C_{1} +C\local )$
\end{lemma}
\begin{proof}
Letting $\eta_{x}^{k,l}$ be the cut-off function as in the previous lemma for $l<k-3$. Let $\tilde{u}=\eta_{x}^{k,l} u -\Qintp(\eta_{x}^{k,l} u )\in \wt{V}^f(\perf\backslash \wt{\omega}_{x,k-l-2})$, and note that from Lemma \ref{qi} we have 
\begin{align}\label{qiestimate}
\TwoNorm{\nabla( \eta_{x}^{k,l}u-\tilde{u})}{\perf }=\TwoNorm{\nabla \Qintp(\eta_{x}^{k,l} u ) }{\perf }\leq C_{1} l^{-1}\TwoNorm{ \nabla u }{ \wt{\omega}_{x,k+2}\backslash \wt{\omega}_{x,k-l-2} },
\end{align}
from this estimate and the properties of $F$ we have 
\begin{align}\label{relation}
\int_{\perf\backslash \wt{\omega}_{x,k-l-2}} \nabla u \nabla \tilde{u}dz=\int_{\perf} \nabla u \nabla \tilde{u}dz=F(\tilde{u})=0.
\end{align}
We have via Caccioppoli type argument that 
\begin{align}
\TwoNorm{\nabla u}{\perf\backslash \wt{\omega}_{x,k}}^2&\leq \int_{\perf\backslash \wt{\omega}_{x,k-l-2}} \eta_{x}^{k,l} |\nabla u|^2dz\\
&\leq  \int_{\perf\backslash \wt{\omega}_{x,k-l-2}} \nabla u\left(\nabla (\eta_{x}^{k,l} u ) -u\nabla \eta_{x}^{k,l}     \right)dz.
\end{align}
Using the fact that $\Qintp(u)=0$, estimate \eqref{qiestimate}, and  the relation \eqref{relation} we have 
\begin{align*}
\TwoNorm{\nabla u}{\perf\backslash \wt{\omega}_{x,k}}^2&\leq  \int_{\perf\backslash \wt{\omega}_{x,k-l-2}} \nabla u(\nabla (\eta_{x}^{k,l} u )-\tilde{u})dz\\ 
&-\int_{\perf\backslash \wt{\omega}_{x,k-l-2}}\nabla u(u-\Qintp(u))\nabla \eta_{x}^{k,l}dz\\
&\leq  C_{1} l^{-1}\TwoNorm{ \nabla u }{ \perf \backslash \wt{\omega}_{x,k-l-2} }^2\\ 
&+C(l H)^{-1}\TwoNorm{ \nabla u }{ \perf \backslash \wt{\omega}_{x,k-l-2} }\TwoNorm{  u -\Qintp(u)}{ \perf \backslash \wt{\omega}_{x,k-l-2} }    \\
&\leq l^{-1} C_{2}\TwoNorm{ \nabla u }{ \perf \backslash \wt{\omega}_{x,k-l-2} }^2.
\end{align*}
On the last term we used the projection estimate \eqref{stableproj} and here  $C_{2}=(C_{1} +C\local ).$ Note here that this $C$ is the benign constant from the estimate of $\nabla \eta_{x}^{k,j}.$ Taking $l=\lceil C_{2} e \rceil$ and successive applications of the above estimate yields 
\begin{align*}
\TwoNorm{\nabla u}{\perf\backslash \wt{\omega}_{x,k}}^2&\leq e^{-1} \TwoNorm{ \nabla u }{ \perf \backslash \wt{\omega}_{x,k-l-2} }^2\\
&\leq e^{- \lfloor \frac{k-1}{l+2}  \rfloor} \TwoNorm{ \nabla u }{ \perf \backslash \wt{\omega}_{x,1} }^2
\leq e^{- \lfloor \frac{k}{l+2}  \rfloor} \TwoNorm{ \nabla u }{ \perf  }^2.
\end{align*}
Finally, taking $\theta=e^{-\frac{1}{\lceil C_{2} e \rceil+2}}$ yields the result. \qed

\end{proof}

We now are ready to state our result on the error introduced from localization. The heart of this argument is to estimate the error between the truncated corrector $Q_{k}$ constructed, after summing over $x$ from \eqref{Qcorrector} and the ideal corrector when $k$ is large enough so that we obtain $Q_{\perf}$. 

\begin{lemma}\label{localglobal}
Let $u_{H}\in \wt{V}_{H}$, let $Q_{m}$ be constructed from \eqref{Qcorrector}, and $Q_{\perf}$ defined to be the "ideal" corrector without truncation, then 
\begin{align}
\TwoNorm{\nabla( Q_{\perf}(u_{H})-Q_{m}(u_{H}))  }{\perf}\leq   m^{\frac{d}{2}}C_{4} \theta^{m  }   \TwoNorm{ \nabla Q_{\perf}(u_{H})  }{\perf},
\end{align}
with  $C_{4}=C_{3}    (1+ C_{1}^2 )^{\frac{1}{2}}$ and $C_{3}=(1+C_{1}+C\local)$.
\end{lemma}
\begin{proof}
Recall that $Q_{m}(u_{H})=\sum_{x\in {\cal N}_{H}  }Q_{x,m}(u_{H})$ with 
\begin{align}
 \int_{\wt{\omega}_{x,m}}\nabla Q_{x,m}(u_{H}) \nabla wdz=\int_{\wt{\omega}_{x}} \hat{\lambda}_{x}\nabla u_{H} \nabla w,dz \text{ for all } w \in \wt{V}^f(\wt{\omega}_{x,m}).
\end{align}
For all $x\in {\cal N}_{H}$, and letting $F_{x}(w):=\int_{\perf}\hat{\lambda}_{x}\nabla u_{H} \nabla wdz$. Note that for $w\in \wt{V}^f(\perf \backslash \wt{\omega}_{x})$, we have $F_{x}(w)=0$. Let $x\in {\cal N}_{H}$ and choose a $x'\in {\cal N}_{H}$ such that $\wt{\omega}_{x'}\cap\wt{\omega}_{x}\neq \emptyset$. We have  $\wt{\omega}_{x}\subset \wt{\omega}_{x',1}$ and so $  \wt{V}^f(\perf \backslash \wt{\omega}_{x',1})\subset  \wt{V}^f(\perf \backslash \wt{\omega}_{x}).$ Thus, $F_{x}$ satisfies the conditions of Lemma \ref{decaylemma}. 

Choosing  $k\geq m$,
we have that $\wt{\omega}_{x',k}\subset \wt{\omega}_{x,m}$.
We denote $v= Q_{\perf}(u_{H})-Q_{m}(u_{H}))\in \wt{V}^f,$ subsequently $\Qintp(v)=0$. Taking the cut-off function $\eta_{x'}^{k,1}$ we have 
\begin{align}
\label{term1}
\TwoNorm{\nabla v}{\perf}^2&=\sum_{x\in {\cal N}_{H}}\int_{\perf}\nabla( Q_{x,\perf}(u_{H})-Q_{x,m}(u_{H}))\nabla (v(1-\eta_{x'}^{k,1}))dz \\
\label{term2}
&+\sum_{x\in {\cal N}_{H}}\int_{\perf}\nabla( Q_{x,\perf}(u_{H})-Q_{x,m}(u_{H}))\nabla (v\eta_{x'}^{k,1})dz.
\end{align}
Estimating the right hand side of \eqref{term1} for each $x$ we have 
\begin{align*}
&\int_{\perf}\nabla( Q_{x,\perf}(u_{H})-Q_{x,m}(u_{H}))\nabla (v(1-\eta_{x'}^{k,1})) dz\\
&\leq \TwoNorm{\nabla( Q_{x,\perf}(u_{H})-Q_{x,m}(u_{H}))}{\perf} \TwoNorm{\nabla (v(1-\eta_{x'}^{k,1}))}{\wt{\omega}_{x',k}}\\
&\leq \TwoNorm{\nabla( Q_{x,\perf}(u_{H})-Q_{x,m}(u_{H}))}{\perf} \left(\TwoNorm{\nabla v}{\wt{\omega}_{x',k}} +  \TwoNorm{v\nabla (1-\eta_{x'}^{k,1}))}{\wt{\omega}_{x',k}\backslash \wt{\omega}_{x',k-1}}   \right)\\
&\leq \TwoNorm{\nabla( Q_{x,\perf}(u_{H})-Q_{x,m}(u_{H}))}{\perf} \left(\TwoNorm{\nabla v}{\wt{\omega}_{x',k}} +  CH^{-1}\TwoNorm{v-\Qintp(v)}{\wt{\omega}_{x',k}\backslash \wt{\omega}_{x',k-1}}   \right)\\
&\leq \TwoNorm{\nabla( Q_{x,\perf}(u_{H})-Q_{x,m}(u_{H}))}{\perf} \left(1 +  C\local   \right)\TwoNorm{\nabla v}{\wt{\omega}_{x',k+1}}.
\end{align*}
As in the proof of Lemma \ref{decaylemma}, $\tilde{v}=\eta_{x'}^{k,1} v -\Qintp(\eta_{x'}^{k,1} v )\in \wt{V}^f(\perf\backslash \wt{\omega}_{x',k-3}).$ 
Letting $m$ be large enough so that $k\geq 4$, then $\tilde{v}\in  \wt{V}^f(\perf\backslash \wt{\omega}_{x})$ and so we have 
\begin{align}
\int_{\perf}\nabla( Q_{x,\perf}(u_{H})-Q_{x,m}(u_{H}))\nabla\tilde{v}dz=0.
\end{align}
We have now the estimate for \eqref{term2} for $x\in {\cal N}_{H}$ using the above identity and \eqref{qiestimate}
\begin{align*}
&\int_{\perf}\nabla( Q_{x,\perf}(u_{H})-Q_{x,m}(u_{H}))\nabla (v\eta_{x'}^{k,1}-\tilde{v})dz\\
&\leq \TwoNorm{ \nabla( Q_{x,\perf}(u_{H})-Q_{x,m}(u_{H}))  }{ \perf  }\TwoNorm{ \nabla (v\eta_{x'}^{k,1}-\tilde{v})  }{ \perf  }\\
&\leq \TwoNorm{ \nabla( Q_{x,\perf}(u_{H})-Q_{x,m}(u_{H}))  }{ \perf  }C_{1}\TwoNorm{ \nabla v }{ \wt{\omega}_{x',k+2 } }
\end{align*}
Combing the estimates for \eqref{term1} and \eqref{term2} we obtain 
\begin{align}
\TwoNorm{\nabla v}{\perf}^2&\leq \sum_{x\in {\cal N}_{H}} \TwoNorm{\nabla( Q_{x,\perf}(u_{H})-Q_{x,m}(u_{H}))}{\perf} \left(1 +  C\local   \right)\TwoNorm{\nabla v}{\wt{\omega}_{x',k+1}}\nonumber\\
&+\sum_{x\in {\cal N}_{H}}\TwoNorm{ \nabla( Q_{x,\perf}(u_{H})-Q_{x,m}(u_{H}))  }{ \perf  }C_{1}\TwoNorm{ \nabla v }{ \wt{\omega}_{x',k+2 } } \nonumber \\
&\leq \sum_{x\in {\cal N}_{H}}\TwoNorm{ \nabla( Q_{x,\perf}(u_{H})-Q_{x,m}(u_{H}))  }{ \perf  }(1+C_{1}+C\local)\TwoNorm{ \nabla v }{ \wt{\omega}_{x',k+2 } }\nonumber\\
\label{vestimate}
&\leq k^{\frac{d}{2}}C_{3} \left(\sum_{x\in {\cal N}_{H}}\TwoNorm{ \nabla( Q_{x,\perf}(u_{H})-Q_{x,m}(u_{H}))  }{ \perf  }^2\right)^{\frac{1}{2}}\TwoNorm{ \nabla v }{\perf },
\end{align}
supposing the $\# \{x\in {\cal N}_{H}|\wt{\omega}_{x}\subset \wt{\omega}_{x',k+2}\}\leq k^{\frac{d}{2}}$, as is guaranteed by quasi-uniformity of the coarse grid.  Here we have $C_{3}=(1+C_{1}+C\local)$.
To estimate $\TwoNorm{ \nabla( Q_{x,\perf}(u_{H})-Q_{x,m}(u_{H}))  }{ \perf  }$ we use the Galerkin orthogonality of the local problem, that is 
\begin{align}\label{galerkinlocal}
\TwoNorm{ \nabla( Q_{x,\perf}(u_{H})-Q_{x,m}(u_{H}))  }{ \perf  }\leq \inf_{q\in \wt{V}^f(\wt{\omega}_{x',k}) } \TwoNorm{ \nabla( Q_{x,\perf}(u_{H})-q)  }{ \perf  }.
\end{align}
Taking $q_{x}=(1-\eta^{k,1}_{x'})Q_{x,\perf}(u_{H})-\Qintp((1-\eta^{k,1}_{x'})Q_{x,\perf}(u_{H})  )\in \wt{V}^f(\wt{\omega}_{x',k})$, we have 
\begin{align*}
\TwoNorm{ \nabla( Q_{x,\perf}(u_{H})-Q_{x,m}(u_{H}))  }{ \perf  }^2&\leq \TwoNorm{ \nabla( \eta^{k,1}_{x'}Q_{x,\perf}(u_{H})-\Qintp((1-\eta^{k,1}_{x'})Q_{x,\perf}(u_{H})  )) }{ \perf  }^2\\
&\leq \TwoNorm{ \nabla Q_{x,\perf}(u_{H}) }{\perf\backslash \wt{\omega}_{x',k-2} }^2+ \TwoNorm{ \nabla(\Qintp((1-\eta^{k,1}_{x'})Q_{x,\perf}(u_{H})  )) }{ \perf  }^2.
\end{align*}
Using Lemma \ref{qi} and Lemma \ref{decaylemma} on the second term we arrive at 
\begin{align*}
\TwoNorm{ \nabla( Q_{x,\perf}(u_{H})-Q_{x,m}(u_{H}))  }{ \perf  }^2 &\leq \TwoNorm{ \nabla Q_{x,\perf}(u_{H}) }{ \perf \backslash \wt{\omega}_{x',k-2} }^2+ C_{1}^2 \TwoNorm{ \nabla Q_{x,\perf}(u_{H})  ) }{ \wt{\omega}_{x',k+2} \backslash \wt{\omega}_{x',k-3}   }^2\\
&\leq (1+ C_{1}^2 )\TwoNorm{ \nabla Q_{x,\perf}(u_{H}) }{ \perf \backslash \wt{\omega}_{x',k-3} }^2\\
&\leq (1+ C_{1}^2 )\theta^{2(k-3)} \TwoNorm{ \nabla Q_{x,\perf}(u_{H}) }{ \perf  }^2\\
&\leq (1+ C_{1}^2 )\theta^{2m } \TwoNorm{ \nabla Q_{x,\perf}(u_{H}) }{ \perf  }^2.
\end{align*}
Combining this estimate into \eqref{vestimate} we arrive at the final estimate that 
\begin{align*}
\TwoNorm{\nabla v}{\perf}&\leq k^{\frac{d}{2}}C_{3} \left(\sum_{x\in {\cal N}_{H}}\TwoNorm{ \nabla( Q_{x,\perf}(u_{H})-Q_{x,m}(u_{H}))  }{ \perf  }^2\right)^{\frac{1}{2}}\\
&\leq k^{\frac{d}{2}}C_{3} \left(\sum_{x\in {\cal N}_{H}}(1+ C_{1}^2 )\theta^{2m  } \TwoNorm{ \nabla Q_{x,\perf}(u_{H}) }{ \perf  }^2\right)^{\frac{1}{2}}\\
&\leq   m^{\frac{d}{2}}C_{4} \theta^{m   }\TwoNorm{ \nabla Q_{\perf}(u_{H}) }{ \perf  }.
\end{align*}
Here we used $Q_{\perf}=\sum_{x\in {\cal N}_{H}}Q_{x,\perf}$ and denoted $C_{4}=C_{3}    (1+ C_{1}^2 )^{\frac{1}{2}}$. \qed

\end{proof}

\section{References}
\bibliographystyle{siam}	
\bibliography{HMP_references.bib}	
\end{document}